\documentclass[10pt]{article}
\newcommand{\version}{\today}

\usepackage{amsthm,amsfonts,amsmath, amscd}
\usepackage{enumerate}
\usepackage{amssymb, amsmath}
\usepackage{mathrsfs}
\usepackage{amscd}
\usepackage[active]{srcltx}
\usepackage{verbatim}

\usepackage[hypertex]{hyperref}

\theoremstyle{plain}
\newtheorem{theorem}{Theorem}[section]

\newtheorem{lemma}[theorem]{Lemma}
\newtheorem{proposition}[theorem]{Proposition}

\newtheorem{definition}[theorem]{Definition}

\newtheorem{example}[theorem]{Example}

\theoremstyle{remark}
\newtheorem{remark}[theorem]{Remark}

\def\N{{\mathbb N}}

\def\R{{\mathbb R}}
\def\C{{\mathbb C}}

\renewcommand{\a}{\alpha}

\newcommand{\Ga}{\Gamma}

\newcommand{\eps}{\varepsilon}

\newcommand{\VV}{\mathbf{V}}

\DeclareMathOperator{\tr}{Tr}

\DeclareMathOperator{\dive}{div}
\DeclareMathOperator{\arctanh}{arctanh}

\DeclareMathOperator{\Hess}{Hess}

\newcommand{\beq}{\begin{equation}}
\newcommand{\eeq}{\end{equation}}
\newcommand{\bal}{\begin{aligned}}
\newcommand{\eal}{\end{aligned}}
\newcommand{\ben}{\begin{enumerate}}
\newcommand{\beni} {\begin{enumerate}[(i)]}
\newcommand{\een}{\end{enumerate}}
\newcommand{\bit}{\begin{itemize}}
\newcommand{\eit}{\end{itemize}}
\newcommand{\beqw}{\begin{equation*}}
\newcommand{\eeqw}{\end{equation*}}
\newcommand{\bthm}{\begin{theorem}}
\newcommand{\ethm}{\end{theorem}}
\newcommand{\bpr}{\begin{proposition}}
\newcommand{\epr}{\end{proposition}}
\newcommand{\ble}{\begin{lemma}}
\newcommand{\ele}{\end{lemma}}
\newcommand{\blem}{\begin{lemma}}
\newcommand{\elem}{\end{lemma}}
\newcommand{\bpf}{\begin{proof}}
\newcommand{\epf}{\end{proof}}
\newcommand{\bex}{\begin{example}}
\newcommand{\eex}{\end{example}}
\newcommand{\bre}{\begin{example}}
\newcommand{\ere}{\end{example}}
\newcommand{\bma}{\begin{bmatrix}}
\newcommand{\ema}{\end{bmatrix}}
\newcommand{\dd}{{\,\rm d}}

\newcommand{\ddt}{\frac{\mathrm{d}}{\mathrm{d}t}}

\newcommand{\ddtt}{\frac{\mathrm{d^2}}{\mathrm{d}t^2}}

\renewcommand{\aa}{{\boldsymbol\alpha}}

\renewcommand{\hat}{\widehat}
\newcommand{\ot}{\otimes}
\newcommand{\hN}{\widehat\Num}
\newcommand{\bip}[1]{\big\langle {#1}\big\rangle_{L^2(\tau)}}

\renewcommand{\phi}{\varphi}

\newcommand{\cA}{\mathcal{A}}
\newcommand{\cB}{\mathcal{B}}
\newcommand{\rr}{\mathbf{r}}
\newcommand{\nn}{\mathbf{n}}
\newcommand{\mm}{\mathbf{m}}

 \newcommand{\diff}[2]{\frac{\mathrm{d}{#1}}{\mathrm{d}{#2}}}

 \renewcommand{\AA}{\mathbf{A}}     
 \newcommand{\BB}{\mathbf{B}}   
 \newcommand{\Cl}{\mathfrak{C}}     
 \newcommand{\Cln}{\Cl^n}           

 \newcommand{\Clo}{\Cl_0}

 \newcommand{\contr}{*}            
 \newcommand{\Ent}{S}               
 \newcommand{\Dens}{{\mathfrak P}}
 \newcommand{\Gr}{\Gamma}           
 \renewcommand{\H}{\mathcal{H}}     
 \newcommand{\ip}[1]{\langle {#1}\rangle_{L^2(\tau)}}   
 \newcommand{\ncd}{\partial}        
 \newcommand{\Num}{\mathcal{N}}     
 \newcommand{\trace}{\tau}          

 \newcommand{\hrho}{\widehat{\rho}}

 \newcommand{\Alg}{\mathcal{A}}

 \newcommand{\uu}{\mathbf{u}} 
 \newcommand{\ipo}[1]{\langle {#1}\rangle}   
 \newcommand{\ipr}[1]{\langle {#1}\rangle_{\rho}}   

\begin{document}

\markboth{\scriptsize{CM \version}}{\scriptsize{CM \version}}

\title
{AN ANALOG OF THE $2$-WASSERSTEIN METRIC IN  
NON-COMMUTATIVE PROBABILITY UNDER WHICH THE
FERMIONIC FOKKER-PLANCK EQUATION IS GRADIENT FLOW
FOR THE ENTROPY}

\author{
\vspace{5pt}  Eric A. Carlen$^{1}$  and Jan Maas$^{2}$ \\
\vspace{5pt}\small{$^{1}$ Department of Mathematics, Hill Center,}\\[-6pt]
\small{Rutgers University,
110 Frelinghuysen Road
Piscataway NJ 08854-8019 USA}\\
\vspace{5pt}\small{$^{2}$ University of Bonn, Institute for Applied Mathematics, }\\[-6pt]
\small{Endenicher Allee 60, 53115 Bonn, Germany.}\\[-6pt]
 }
 
\date{\today}
\maketitle

\footnotetext [1]{Work partially supported by U.S.
National Science Foundation grant  DMS 0901632.}

\footnotetext [2]{Work partially supported by Rubicon subsidy 680-50-0901 of the Netherlands Organisation for Scientific Research (NWO). \hfill\break
\copyright\, 2012 by the authors. This paper may be reproduced, in its
entirety, for non-commercial purposes.
}



 \begin{abstract}
Let $\Cl$ denote the Clifford algebra over $\R^n$, which is the von Neumann algebra generated by $n$ self-adjoint operators $Q_j$,  $j=1,\dots,n$ satisfying the
canonical anticommutation relations, $Q_iQ_j+Q_jQ_i = 2\delta_{ij}I$, and let $\tau$ denote the normalized trace on $\Cl$. This  algebra arises in
quantum mechanics as the algebra of observables generated by $n$ Fermionic degrees of freedom.  Let $\Dens$ denote the set of all positive operators 
$\rho\in\Cl$ such that $\tau(\rho) =1$; these are the non-commutative analogs of probability densities in the non-commutative probability space
$(\Cl,\tau)$. The Fermionic Fokker-Planck equation is a quantum-mechanical analog of the classical Fokker-Planck equation with which it has much in
common, such as the same optimal hypercontractivity properties. In this paper we construct a Riemannian metric on $\Dens$ that we show to be a natural analog
of the classical $2$-Wasserstein metric, and we show that, in analogy with the classical case, the Fermionic Fokker-Planck equation is gradient flow in this metric for 
the relative entropy with respect to the ground state. We derive a number of consequences of this, such as a sharp Talagrand inequality for this metric, and we prove a number of results pertaining to this metric. Several open problems are raised. 
 \end{abstract}


\tableofcontents

\maketitle

\section{Introduction}

Many partial differential equations for the evolution of classical probability densities $\rho(x,t)$ on $\R^n$  can be viewed as 
describing gradient flow with respect to the $2$-Wasserstein metric. This point 
of view is due to Felix Otto, and he and others have shown it to be remarkably effective for gaining quantitative control over the behavior
of such evolution equations. We recall that for two probability densities $\rho_0$ and $\rho_1$ on $\R^n$, both with finite second moments,
the set of {\em couplings} ${\mathcal C}(\rho_0,\rho_1)$ is the set of all probability measures $\kappa$ on $\R^{2n}$ such that for all test functions
$\varphi$ on $\R^n$,
$$\int_{\R^{2n}}\varphi(x){\rm d}\kappa(x,y) = \int_{\R^n} \varphi(x)\rho_0(x){\rm d}x $$
and
$$\int_{\R^{2n}}\varphi(y){\rm d}\kappa(x,y) = \int_{\R^n} \varphi(y)\rho_1(y){\rm d}x\ .$$
That is, a probability measure ${\rm d}\kappa$ on the product space $\R^{2n}$ is in ${\mathcal C}(\rho_0,\rho_1)$ if and only if the first and second 
marginals of ${\rm d}\kappa$ are $\rho_0(x){\rm d}x$ and   $\rho_1(y){\rm d}y$ respectively.
Then the $2$-Wasserstein distance between $\rho_0$ and $\rho_1$, ${\rm W}(\rho_0,\rho_1)$,  is defined by
\begin{equation}\label{Kant}
{\rm W}^2(\rho_0,\rho_1) = \inf_{\kappa\in  {\mathcal C}(\rho_0,\rho_1)}\int_{\R^{2n}}\frac12 |x-y|^2 {\rm d}\kappa(x,y)\ .
\end{equation}
One  may  view the conditional distribution of $y$ under $\kappa$ given $x$,  which is
$\rho_0(x)^{-1}\kappa(x,y){\rm d}y$  if $\kappa$ has a density $\kappa(x,y)$, as a ``transportation plan''  specifying  to
 where the mass at $x$ gets transported, and in what proportions, in a transportation process transforming the mass distribution $\rho_0(x){\rm d}x$ into 
 $\rho_1(y){\rm d}y$.
 The function $|x-y|^2/2$ is interpreted as giving the  cost of moving a unit of mass from $x$ to $y$, and then the  minimum total cost, considering all possible ``transportation plans'',
is the square of the Wasserstein distance. For details and background, see \cite{Vil03}.

In quantum mechanics, classical probability densities are replaced by quantum mechanical density matrices; i.e., 
positive trace class operators $\rho$ on some Hilbert space such that $\tr(\rho) =1$.  These are the analogs of probability densities within the context of
non-commutative probability theory originally due to Irving Segal \cite{S53,S56,S65}. The starting point of his generalization of classical probability theory is the fact that the set of
all complex bounded functions that are measurable with respect to some $\sigma$-algebra, equipped with the complex conjugation as the involution $\ast$, 
form a commutative von Neumann algebra, and any probability measure on this measurable space induces a positive linear functional; i.e., a {\em state}
on the algebra.  In Segal's generalization, one drops the requirement that the von Neumann algebra be commutative. The resulting non-commutative probability
 spaces -- von Neumann algebras with a specified state -- turn out to have many uses, particularly in quantum mechanics, where the $L^2$ spaces built on
 them give a convenient representation of the operators  relevant to the analysis of many physical systems. We shall 
 discuss one example of this in detail below.

If the von Neumann algebra in question is ${\mathcal B}({\mathcal H})$,
the set of all abounded operators on the Hilbert space ${\mathcal H}$,  there is  no obvious non-commutative analog of the $2$-Wasserstein metric. 
One can generalize the notion of a coupling of two density matrices
$\rho_0,\rho_1$ on a Hilbert space ${\mathcal H}$ to be a  density matrix $\kappa$ on ${\mathcal H}\otimes {\mathcal H}$ whose
{\em partial traces} over the second and first factor are $\rho_0$ and $\rho_1$ respectively. Based on this idea, an analog of the Wasserstein metric has been defined by Biane and Voiculescu in the setting of free probability \cite{BV01}. However, in general there is no natural analog of the conditioning operation so that in the general quantum case, there is no natural way to decompose a coupling, via conditioning, into a transportation plan. Moreover, since there is no underlying metric space, there is no obvious analog of the cost function $|x-y|^2/2$. 

However, there are physically interesting evolution equations for density matrices that are close quantum mechanical relatives of classical equations
for which the Wasserstein metric point of view has proven effective. This fact suggests that at least in certain particular non-commutative probability spaces of relevance to
quantum mechanics, there should be a meaningful analog of the $2$-Wasserstein metric.  As we shall demonstrate here, this is indeed the case.

The prime example of such an evolution equation is the {\em Fermionic Fokker-Planck Equation} introduced by Gross \cite{G72,G75}. 
As we explain below, this equation describes the evolution of density matrices belonging to the  {\em operator algebra
generated by $n$ Fermionic degrees of freedom} which  turns out to be a {\em Clifford algebra}.  
In this operator algebra, there is also a differential calculus, and Gross showed that using the operators pertaining to this differential
calculus, one can write the  Fermionic Fokker-Planck Equation in a form that displays it as an almost ``identical twin'' of the classical Fokker-Planck equation.  

As an example of the close parallel between the classical and Fermionic Fokker-Planck equations, consider one of the most significant properties of the evolution described
by the classical equation is its hypercontractive property, expressed in Nelson's sharp hypercontractivity inequality \cite{N73}.  
The exact analog of Nelson's sharp hypercontractivity inequality for the classical Fokker-Planck evolution has been shown to hold for the Fermionic Fokker-Planck 
evolution \cite{CL1},  where it involves non-commutative analogs of the $L^p$ norms in the (non-commutative) operator algebra generated by $n$ Fermionic degree of freedom.

Other significant features of the classical Fokker-Planck evolution have lacked a quantum counterpart. For instance,
as shown by Jordan, Kinderlehrer and Otto \cite{JKO98}, the classical Fokker-Planck Equation for $\rho(x,t)$  is gradient flow in the $2$-Wasserstein metric of  the relative relative entropy of $\rho(x,t)$ with respect to the equilibrium Gaussian measure. Moreover, crucial properties of this evolution, such as its hypercontractive properties, can be deduced from the convexity properties of the relative entropy functional in the $2$-Wasserstein metric. A similar gradient flow structure in the space of probability measures has meanwhile been developed and exploited in many different settings \cite{AGS08,AGS11a,CG04,CMV06,CHLZ11,Erb10,FSS10,GKO10,Ma11,Mie11a,Mie11b,OhSt09,O01}.

The purpose of our paper is to   construct  a non-commutative analog of the $2$-Wasserstein metric,
and to prove  a number of results concerning this metric that further the parallel between the quantum and classical cases.
The first step will be to  construct  the metric, and here, a judicious choice of the point of departure is crucial.
Among the many equivalent ways to define the Wasserstein metric, the one that seems most useful in the non-commutative setting is the dynamical approach of Benamou and Brenier \cite{BB00}. 
In their approach, couplings are defined not in terms of joint probability measures, but in terms of smooth paths $t\mapsto \rho(x,t)$ in the space of probability densities.  Any such path satisfies the continuity equation
\begin{equation}\label{cont}
\frac{\partial}{\partial t} \rho(x,t) + {\rm div}[{\bf v}(x,t)\rho(x,t)] = 0
\end{equation}
for some time dependent vector field ${\bf v}(x,t)$.  A pair $\{\rho(\cdot,\cdot),{\bf v}(\cdot,\cdot)\}$ is said to {\em couple}  $\rho_0$ and $\rho_1$ provided that the pair satisfies (\ref{cont}), $\rho(x,0) = \rho_0(x)$ and $\rho(x,1) = \rho_1(x)$.  
Using the same symbol  ${\mathcal C}(\rho_0,\rho_1)$ to denote the set of couplings between $\rho_0$ and $\rho_1$ in this new sense, 
Benamou and Brenier show that ${\rm W}(\rho_0,\rho_1)$ is given by
\begin{equation}\label{BEBR}
{\rm W}^2(\rho_0,\rho_1) = 
 \inf_{\{\rho,{\bf v}\}\in  {\mathcal C}(\rho_0,\rho_1)}  \frac12\int_0^1 \int_{\R^n}  |{\bf v}(x,t)|^2\rho(x,t)\dd x\dd t\  .
\end{equation}
Moreover, they showed how one can characterize the {\em geodesic paths} for the $2$-Wasserstein metric in terms of solutions of a Hamilton-Jacobi equation, and how this 
characterization of the geodesic paths provides an effective means of investigating the convexity properties of functionals on the space of probability densities with respect to the $2$-Wasserstein metric. 

We may now roughly describe our main results:  Working in an operator algebra setting in which there exists a differential calculus, and hence a divergence, we develop a non-commutative analog of the continuity equation (\ref{cont}) and show how this leads to a non-commutative analog of the Benamou-Brenier formula for the $2$-Wasserstein difference.  Actually, since there are many ways one might try to generalize (\ref{cont}) to the non-commutative setting, we start out by computing a formula for the dissipation of the relative entropy along the Fokker-Planck evolution, and use this to guide us to a suitable generalization of (\ref{cont}).  

With a suitable continuity equation in hand, we proceed to the definition of our Riemannian metric, and prove that the Fermionic Fokker-Planck evolution is gradient flow for the relative entropy
with respect to the ground state in this metric. 
The rest of the paper is then devoted to an investigation of the properties of this new metric. We note that the operator algebra we consider is finite dimensional, and so the metric we investigate is a {\em bona-fide} Riemannian metric. Among our other results, using the known sharp logarithmic Sobolev inequality for the  Fermionic Fokker-Planck equation \cite{CL1}, 
we deduce a sharp Talagrand-type inequality for our metric. 

We begin by recalling some useful background material on the classical and Fermionic Fokker-Planck equations.

\section{The classical and Fermionic  Fokker-Planck equations}

\subsection{The classical Fokker-Planck equation}

The classical  Fokker-Planck equation is 
\begin{equation}\label{FP1}
\frac{\partial}{\partial t} f(t,x) = \nabla\cdot(\nabla + x) f(t,x)\;,
\end{equation}
where $f(x,t)$ is a time dependent probability density on $\R^n$.  Note that the standard Gaussian probability density
\begin{equation}\label{FP2}
 \gamma_n(x) := (2\pi)^{-n/2}e^{-|x|^2/2}
 \end{equation}
 is a steady-state solution. 
 
 Let $f(x,t)$ be a solution of (\ref{FP1}), and define a function $\rho(x,t)$ by
\begin{equation}\label{FP3} 
f(x,t) = \rho(x,t)   \gamma_n(x)\ .
\end{equation}
Then $\rho(x,t)$ satisfies  
\begin{equation}\label{FP4}
\frac{\partial}{\partial t} \rho(t,x) = (\nabla -x)\cdot \nabla \rho(x,t)\;.
\end{equation}
The solution of the Cauchy problem for (\ref{FP3}) with initial data $\rho_0(x)$ is given by {\em Mehler's formula}
\begin{equation}\label{FP5}
\rho(x,t) = \int_{\R^n} \rho_0\left(e^{-t}x + (1- e^{-2t})^{1/2}y\right)\gamma_n(y)\dd y\ .
\end{equation}
(A simple computation shows that (\ref{FP5}) does indeed define the solution of (\ref{FP4}) with the right initial data.)

The Mehler semigroup is the  semigroup on $L^2(\R^n,\gamma_n(x){\rm d}x)$ consisting of the operators 
$$P_t\varphi(x) = \int_{\R^n} \varphi\left(e^{-t}x + (1- e^{-2t})^{1/2}y\right)\gamma_n(y)\dd y\ .$$
Each of these operator is {\em Markovian}; i.e., positivity preserving with $P_t1=1$. Hence the Mehler semigroup is a Markovian semigroup and 
the associated {\em Dirichlet form} is the non-negative quadratic form
\begin{equation}\label{BOS}
{\mathcal B}(\varphi,\varphi) := \lim_{t\to0}\frac{1}{t} \int_{\R^n} \varphi(x) [\varphi(x) - P_t\varphi(x)] \gamma_n(x){\rm }\dd x =  \int_{\R^n}|\nabla \varphi(x)|^2 \gamma_n(x)\dd x\ .
\end{equation}
The positive operator
$$N := -(\nabla - x) \cdot\nabla $$
satisfies
$${\mathcal B}(\psi,\varphi) = \langle  \psi,N\varphi\rangle_{L^2(\gamma_n\dd x)} $$
for all smooth, bounded $\psi$ and $\varphi$, and then the domain of self-adjointness is given by the Friedrich's extension. 
The spectrum of $N$ consists of the non-negative integers;
 its eigenfunctions are the Hermite polynomials. Since the corresponding eigenvalue is the degree of the  Hermite polynomial,
the operator $N$ is sometimes referred to as the {\em number operator}. By what we have said above, $N$ is the generator of the Mehler semigroup; i.e., 
   $P_t := e^{-tN}$, $t \geq 0$.

There is a close connection between the  Fokker-Planck equation and entropy. 
Given a probability density $f(x)$ with respect to Lebesgue measure on $\R^n$, the {\em relative entropy of $f$ with respect to $\gamma_n$}
is the quantity $H(f|\gamma_n)$ defined by
\begin{align*}
H(f|\gamma_n) &= \int_{\R^n}\left(\frac{f}{\gamma_n}\right)\log  \left(\frac{f}{\gamma_n}\right) \gamma_n(x)\dd x\\
&=  \int_{\R^n} f \log f (x) \dd x + \frac12 \int_{\R^n} |x|^2 f(x) \dd x  + \frac{n}{2} \log(2\pi)\ .
\end{align*}
Notice that if $f(x) =  \rho(x)\gamma_n(x)$, then
$$
H(f|\gamma_n)  = \int_{\R^n} \rho(x)\log \rho(x) \gamma_n(x)\dd x\ .$$

As we have mentioned above,  it has been shown relatively recently 
by  Jordan, Kinderlehrer and Otto \cite{JKO98} that the Fokker-Planck equation may be viewed as the gradient flow of
the relative entropy  with respect to the reference measure $\gamma_n(x)\dd x$ when the space of probability measures on $\R^n$
is equipped with a Riemannian structure induced by the $2$-Wasserstein metric, and further work has shown that many properties of the classical Fokker-Planck evolution can be deduced from the strict uniform convexity  of the relative entropy function along the geodesics  for the $2$-Wasserstein metric (see, e.g., \cite{AGS08, Vil09}). 

To explain the close connection between the classical Fokker-Planck equation, entropy, and the $2$-Wasserstein metric, we first write the Fokker-Flanck equation (\ref{FP1})
as a continuity equation. Note that   (\ref{FP1}) can be written as
\begin{equation}\label{FP1A}
\frac{\partial}{\partial t} f(t,x) + {\rm div}[ f(t,x){\bf v}(x,t)] = 0
\end{equation}
where
\begin{equation}\label{FP1B}
{\bf v}(x,t) = -\nabla \log (f(x,t)) - x
\end{equation}
To see that this choice of ${\bf v}(x,t)$ is consistent with (\ref{FP1A}),  write the time derivative of $f(x,t)$ as the divergence of a vector field, and then  divide this vector field by $f(x,t)$ to obtain the vector field ${\bf v}(x,t)$.

Given a solution $f(x,t)$ of (\ref{FP1}), there are many vector fields $\widetilde {\bf v}(x,t)$ such that  
\begin{equation}\label{FP1AA}
\frac{\partial}{\partial t} f(t,x) + {\rm div}[ f(t,x)\widetilde{\bf v}(x,t)] = 0\ ,
\end{equation}
but the choice made in (\ref{FP1B}) is special since
$$ \int_{\R^n} |{\bf v}(x,t)|^2 f(x,t){\rm d}x <  \int_{\R^n} |\widetilde {\bf v}(x,t)|^2 f(x,t){\rm d}x$$
for any other vector field  $\widetilde {\bf v}(x,t)$ satisfying (\ref{FP1AA}) for our given solution $f(x,t)$. 
Indeed, the set ${\mathcal K}$ of vector fields $\widetilde {\bf v}$ such that (\ref{FP1AA}) is satisfied is a closed convex set in the obvious Hilbertian
norm, and thus there is a unique norm-minimizing element ${\bf v}_0$. Considering perturbations of  ${\bf v}_0$ of the form
${\bf v}_0 + \epsilon f^{-1}{\bf w}$ where ${\bf w}(x,t)$ is, for each $t$, a smooth compactly supported divergence free vector field, one sees that ${\bf v}_0$ must 
satisfy
$$\int_{\R^n} {\bf v}_0(x,t){\bf w}(x,t){\rm d}x = 0$$
for each $t$, and thus, that ${\bf v}_0(x,t)$ is, for each $t$, a gradient. One then shows that there is only one gradient vector field in ${\mathcal K}$, and hence,
since the vector field ${\bf v}(x,t)$ given in (\ref{FP1B}) is a gradient, it is the minimizer. We only sketch this argument here since we will give all of the details of the analogous
argument in the non-commutative setting shortly.  For further discussion in the classical case, see \cite{CG03}.

Now, from the Benamou-Brenier formula for the Wasserstein distance, and the minimizing property of the vector field ${\bf v}(x,t)$ given in (\ref{FP1B}), we see that
$${\rm W}^2(f(\cdot,t),f(\cdot,t+h)) = \left(\frac12 \int_{\R^n} |{\bf v}(x,t)|^2f(x,t){\rm d}x\right)h^2 +  o(h^2)\ .$$

Next, we compute, using the continuity equation form of the Fokker-Planck equation,
\begin{eqnarray}\label{FP1D}
\frac{{\rm d}}{{\rm d}t}  H(f|\gamma_n) &=& -\int_{\R^n} \left[\log f(x,t) + \frac12 |x|^2 \right]  {\rm div}[ f(t,x){\bf v}(x,t)] {\rm d}x \nonumber\\
&=&\phantom{-} \int_{\R^n} \left[\nabla \log f(x,t) + x \right]  [ f(t,x){\bf v}(x,t)] {\rm d}x \nonumber\\
&=& -  \int_{\R^n} |{\bf v}(x,t)|^2 f(x,t){\rm d}x\ .
\end{eqnarray}

In summary, for solutions $f(x,t)$ of the classical Fokker-Planck equation, one has
\begin{equation}\label{EW}
\frac{{\rm d}}{{\rm d}t}  H(f|\gamma_n) = -\left(\lim_{h\to 0}  \frac{  {\rm W}(f(\cdot,t),f(\cdot,t+h))}{h} \right)^2 \ .
\end{equation}
When we come to the non-commutative case, it will not be so evident how to rewrite the Fermionic Fokker-Planck  equation in continuity equation form. The logarithmic gradient of $f(x,t)$ 
enters in (\ref{FP1B})  because we divided by $f(x,t)$ in the course of deducing the formula (\ref{FP1B}) for ${\bf v}(x,t)$.  In the non-commutative case this division must be done in a rather 
indirect way to achieve the desired result, and we shall arrive at the appropriate division formula by working backwards from a calculation of entropy dissipation.

First, we  introduce the Fermionic Fokker-Planck equation, beginning with a brief  introduction to 
Clifford algebras as non-commutative probability spaces. 

\subsection{The Clifford algebra as a non-commutative probability space}

Let $\H$ be a complex Hilbert space and let $Q_1, \ldots, Q_n$ be bounded operators on $\H$ satisfying the canonical anticommutation relations (CAR)
 \begin{equation}\label{CAR}
  Q_i Q_j + Q_j Q_i = 2 \delta_{ij}I\;.
 \end{equation}
The Clifford algebra $\Cl$ is the operator algebra generated by $Q_1, \ldots, Q_n.$ We say ``the'' Clifford algebra because any two 
realizations are unitarily equivalent. We give a brief introduction to $\Cl$ here. Though fairly self-contained for our purposes, we refer
to \cite{CL1} for more detail and further references.

One  realization of $\Cl$ as an operator algebra may be achieved on the Hilbert space $\H$ that is the 
$n$-fold tensor product of $\C^2$ with itself.  Let
$$Q := \left[\begin{array}{cc} 0 & 1\\ 1 & 0\end{array}\right] \quad{\rm and}\qquad U :=   \left[\begin{array}{cc} 1 & \phantom{-}0\\ 0 & -1\end{array}\right] \ .$$
Then let $Q_j$ the be tensor product of the form
$$X_1\otimes X_2 \cdots \otimes X_n\;,$$
where $X_j = Q$, where $X_{i} = U$ for all $i < j$, and where $X_k = I$, the $2\times 2$ identity matrix, for all $k> j$. Then one readily verifies that the 
canonical anti-commutation relations are satisfied.

There is a natural injection of $\R^n$ into $\Cl$ given by 
\begin{equation}\label{CAN}
x \mapsto J(x) :=  \sum_{j=1}^n x_jQ_j\ .
\end{equation}
One then sees, as a consequence of (\ref{CAR}) that 
$J(x)^2 = |x|^2 I$, which is often taken as the  relation defining $\Cl$.

Let $\tau$ denote the 
normalized trace  on $\Cl$.  That is, if $A$ is any operator on $\H$ belonging to $\Cl$, 
$$\tau(A) = 2^{-n}\tr(A)\ .$$
Evidently if $A$ is positive in $\Cl$, meaning that $A$ has positive spectrum, or what is the same, $A = B^*B$ with $B$ in $\Cl$, then
$\tau(A)\geq 0$. Also evidently $\tau(I) = 1$ where $I$ is the identity in $\Cl$. Thus, $\tau$ is a {\em state} on $\Cl$. It may appear that $\tau$ depends on the particular representation
of the CAR that we are employing but this is not the case:  

An $n$-tuple $\aa = (\a_1, \ldots, \a_n) \in \{0,1\}^n$ is called a Fermionic multi-index. We set $|\aa| := \sum_{j=1}^n \a_j$ and
 \begin{align*}
  Q^\aa := Q_1^{\a_1} \cdots Q_n^{\a_n}\;.
 \end{align*}
One readily verifies that
 \begin{equation}\label{tauun}
   \trace (Q^\aa) = \delta_{0,|\aa|}\;.
 \end{equation} 
Since the $\{Q^\aa\}$ are a basis for $\Cl$, there is at most one state, namely $\tau$,  that satisfies (\ref{tauun}).

As emphasized by Segal \cite{S53,S56,S65}, $(\Cl,\tau)$
is an example of a non-commutative probability space that is a close analog of the standard Gaussian probability space
$(\R^n,\gamma_n(x)\dd x)$ where 
$$ \gamma_n(x) := (2\pi)^{-n/2}e^{-|x|^2/2}\ .$$
For instance, a characteristic property of isotropic Gaussian probability measures on $\R^n$ is that if $V$ and $W$ are two orthogonal subspaces of $\R^n$, and $f$ and $g$ are two functions
on $\R^n$ such that $f(x)$ depends only on the component of $x$ in $V$ and  $g(x)$ depends only on the component of $x$ in $W$, then
\begin{equation}\label{IND}
\int_{\R^n} f(x)g(x)\gamma_n(x){\rm d}x = \left(\int_{\R^n} f(x)\gamma_n(x){\rm d}x\right)\left(\int_{\R^n} g(x)\gamma_n(x){\rm d}x\right)\ .
\end{equation}
That is, under an isotropic Gaussian probability law on $\R^n$, random variables generated by orthogonal subspaces of $\R^n$ are statistically independent, and as is well known, this property 
is characteristic of isotropic Gaussian laws. 

In the case of the Clifford algebra, let  $V$ and $W$ be orthogonal subspaces of $\R^n$, and let $\Cl_V$ and $\Cl_W$, respectively, be the subalgebras of $\Cl$ generated by
$J(V)$ and $J(W)$. Then it is easy to see that if $A\in \Cl_V$ and $B\in \Cl_W$, then
$$\tau(AB) = \tau(A)\tau(B)\ , 
$$the analog of (\ref{IND}).

 \subsection{Differential calculus on the Clifford algebra}

The Clifford algebra becomes a Hilbert space endowed with the inner product
 \begin{align*}
  \ip{A, B} := \trace(A^*B)\;, \qquad A, B \in \Cl\;.
 \end{align*} 
The $2^n$ operators $(Q^\aa)_{\aa \in  \{0,1\}^n}$ form an orthonormal basis for $\Cl.$

For $i = 1, \ldots, n,$ we define the partial derivative  by
 \begin{align*}
  \nabla_i(Q^\aa) := \left\{ \begin{array}{ll}
   Q_i Q^\aa, \quad & \a_i = 1\;, \\ 
   0,        \quad & \a_i = 0\;, 
\end{array} \right.
 \end{align*}
and linear extension. We will also consider the gradient
 \begin{align*}
  \nabla : \Cl \to \Cln\;,
  \qquad A \mapsto \big(\nabla_1(A), \ldots, \nabla_n(A)\big)\;,
 \end{align*}
It is easy to check that
 \begin{align*}
  \nabla_i(A) = \frac12 (Q_i A - \Gr(A)Q_i)\;, \qquad A \in \Cl\;,
 \end{align*}
where $\Gr$ denotes the grading operator defined by
 \begin{align*}
  \Gr(Q^\aa) := (-1)^{|\aa|} Q^\aa\;.
 \end{align*}
For $A, B \in \Cl$ the product rule
\begin{align}\label{eq:prod-rule}
 \nabla_i(AB) = \Gr(A)\nabla_i(B) + \nabla_i(A) B
\end{align}
holds, and the following identities are readily checked:
 \begin{align}
  \Gr(A B) & = \Gr(A) \Gr(B)\;, \label{eq:gr-alghom}       \\
  \Gr(A^*) & = \Gr(A)^*\;, \label{eq:gr-adj} \\
  \trace(\Gr(A)B)  & = \trace(A\Gr(B))\;, \label{eq:gr-selfadj} \\
  (\nabla(A^*))^*  & = \Gr(\nabla A)
     = - \nabla(\Gr(A)) \label{eq:grad-adj}\;.
 \end{align}
By (\ref{eq:gr-alghom}) and (\ref{eq:gr-adj}), $A\mapsto \Gamma(A)$ is a $*$-automorphism, and it is often called the {\em principle automorphism} in $\Cl$. 
 
Here, and throughout the rest of this work, we use the convention that for $\AA = (A_1, \ldots, A_n)\in \Cln$ and $B \in \Cl,$
 \begin{align*}
  \AA B := (A_1 B, \ldots, A_n B)\;, \qquad
  B\AA  := (B A_1, \ldots, B A_n )\;.
 \end{align*}
Similarly, we will also extend an operator $T$ acting on $\Cl$ to an operator on $\Cln$ in the obvious way, by defining
 \begin{align*}
  T \AA := (T (A_1), \ldots, T(A_n))\;.
 \end{align*}
 
The adjoint of $\nabla_i$ with respect to the $L^2(\tau)$-inner product is given by
\begin{align*}
 \nabla_i^*(A) = \frac12 (Q_i A + \Gr(A)Q_i)\;, \qquad A \in \Cl\;,
\end{align*}
It follows that
\begin{align*}
  \nabla_i^*(Q^\aa) := \left\{ \begin{array}{ll}
 0,  \quad & \a_i = 1\;, \\ 
   Q_i Q^\aa\;,        \quad & \a_i = 0\;, 
\end{array} \right.
\end{align*}
and the identities
\begin{align} \label{eq:div-adj}
   (\nabla_i^*(A^*))^*
   & = -\Gr(\nabla_i^*A)
     = \nabla_i^*(\Gr(A))\;,
\end{align}
hold.
As usual, the divergence operator is defined by 
\begin{align*}
 \dive (\AA) := - \sum_{i=1}^n \nabla_i^*(A_i)\;.
\end{align*}
 
 \subsection{The Fermionic Fokker-Planck equation}

As noted above, an element $A$ of $\Cl$ is {\em non-negative} if for some $B\in \Cl$, $A = B^*B$. 
An element $A$ of $\Cl$ is {\em strictly positive} if for some $B\in \Cl$ and some $\lambda>0$, $A = B^*B + \lambda I$.
Let $\Dens$ denote the set of (non-commutative) probability densities, i.e., all non-negative elements $\rho \in \Cl$ satisfying $\trace(\rho) = 1.$
Let $\Dens_+$ denote  the set of strictly positive probability densities. 
The Fermionic Fokker-Planck Equation is an evolution equation for probability densities in $\Cl$ that we now define, starting from an analog
of the Dirichlet form (\ref{BOS}) associated to the classical Fokker-Planck equation.
 
Gross's Fermionic Dirichlet form ${\mathcal F}(A,A)$ on $\Cl$ is defined by
$${\mathcal F}(A,A) = \tau\left( (\nabla A)^*\cdot \nabla A\right) = \sum_{j=1}^n \tau \left( (\nabla_j A)^*\cdot \nabla_j A\right)\ .$$
In so far as $\tau$ is an analog of integration against $\gamma_n(x){\rm d}x$, this is a direct analog of (\ref{BOS}).

The {\em Fermionic number operator} ${\mathcal N}$ is defined by
$${\mathcal F}(B,A) = \ip{B, {\mathcal N}A}\ ,$$
and the {\em Fermionic Mehler semigroup} is given by 
$${\mathcal P}_t = e^{-t{\mathcal N}}\ ,$$
for $t\geq 0$. 
On the basis of the connection between the Mehler semigroup and the classical Fokker-Planck equation, we refer to 
\begin{equation}\label{FFPE}
\frac{\partial}{\partial t} \rho(t) =  - {\mathcal N}\rho(t)\ .
\end{equation}
More precisely, this is a direct analog of (\ref{FP4}), the classical Fokker-Planck equation for the evolution of a density with respect to the Gaussian reference measure
$\gamma_n(x){\rm d}x$, instead of with respect to Lebesgue measure, as there is no analog of Lebesgue measure in the quantum non-commutative setting.

At this point it is not obvious that  ${\mathcal P}_t \rho \in \Dens$ whenever $\rho\in \Dens$. Since ${\mathcal N}I = 0$, it is easy to see that $\tau({\mathcal P}_t \rho) = \tau(\rho)$
for all $t$, but the positivity is less evident. One way to see this is through an analog of Mehler's formula that is valid for the  Fermionic Mehler semigroup; see \cite{CL1}.

 \section{The continuity equation in the Clifford algebra and the Riemannian metric}
 
 We are finally finished with preliminaries and ready to 
begin our investigation. If we are to show that
 the Fermionic Fokker-Planck evolution is gradient flow for the relative entropy,
  it must at least be the case that relative entropy is dissipated along this evolution. We start by deducing a formula for the rate of 
  dissipation, and proceed from there to a study of the continuity equation in $\Cl$.

\subsection{Entropy dissipation along the Fermionic Fokker-Planck evolution}

For $\rho \in \Dens$, we define the relative entropy of $\rho$
with respect to $\tau$ to be 
$$S(\rho) = \tau[\rho \log \rho]\ .$$

Given $\rho_0 \in \Dens$, define $\rho_t := {\mathcal P}_t \rho_0$. Then
\begin{eqnarray}\label{basic}
\frac{{\rm d}}{{\rm d}t}S(\rho_t) &=& -\tau\left[ \log \rho_t  {\mathcal N}\rho_t\right]\nonumber\\
&=& - \tau\left[ (\nabla \log \rho_t)^* \cdot \nabla \rho_t\right]\ .
\end{eqnarray}

Our first goal is to rewrite this as the negative of a complete square analogous to (\ref{FP1D}), with the hope of identifying, through this computation, the
form of the ``minimal'' vector field in a continuity equation representation of the Fermionic Fokker-Planck equation. 
We use the following lemma:

\begin{lemma}\label{bas}
For any $\rho\in \Dens_+$, and any index $i$, 
\begin{equation}\label{basic2}
\nabla_i \rho = \int_0^1 \Gamma(\rho)^{1-s}\left[ \nabla_i \log \rho\right]  \rho^s \dd s\ .
\end{equation}
\end{lemma}

\begin{proof} Since $\rho\in \Dens_+$, 
$$\rho= \lim_{k\to\infty}\left(I + \frac1k \log \rho\right)^k\ ,$$
and by the product rule \eqref{eq:prod-rule},
$$\nabla_i \left(I + \frac1k \log \rho\right)^k = \sum_{\ell = 0}^{k-1} \frac1k \Gamma \left(I + \frac1k \log \rho\right)^\ell \left[  \nabla_i \log \rho\right] 
 \left(I + \frac1k \log \rho\right)^{k-\ell-1} \ .$$
The result follows upon taking limits. 
\end{proof}

\begin{remark} It is possible to develop a systematic chain rule for $\nabla_i$, but this simple example is all we need at present.
\end{remark}

Combining (\ref{basic})  and (\ref{basic2}), we obtain
\begin{equation}\label{basic3}
\frac{{\rm d}}{{\rm d}t}S(\rho_t)   = -\tau \left[ (\nabla \log \rho_t)^* \cdot  \int_0^1 (\Gamma \rho_t)^{1-s}\left[ \nabla \log \rho_t\right]  \rho_t^s \dd s\right]\ .
\end{equation}

The formula \eqref{basic2} is the analog of the classical formula $\nabla f(x) = f(x)\nabla \log f(x)$. It suggests that the meaningful analog of dividing by $\rho$ in $\Cl$ will involve
inversion of the operation 
$$C \mapsto  \int_0^1 \Gamma(\rho)^{1-s} C  \rho^s \dd s$$
in $\Cl$. This brings us to the following definition:

\begin{definition} Given strictly positive $m\times m$ matrices $A$ and $B$, define the linear transformation $(A,B)\#$ from the space of $m\times m$ matrices into itself by
\begin{equation}\label{trans1}
(A,B)\#C = \int_0^1 A^{1-s} C B^s \dd s\ .
\end{equation}
\end{definition}

The next theorem is not original, but as we lack a ready reference, we provide the short proof. We note that the $A=B$ case is used in \cite{L}.

\begin{theorem}\label{inversion} Let $A$ and $B$ be strictly positive definite $m\times m$ matrices. Then the linear transformation
$(A,B)\#$ from the space of $m\times m$ matrices into itself is invertible, and if  $(A,B)\#C = D$, then
\begin{equation}\label{trans2}
C = \int_0^\infty (A+xI)^{-1} D   (B+xI)^{-1} \dd x \ .
\end{equation}
\end{theorem}

\begin{proof}
Let $\{u_1,\dots, u_m\}$ be an orthonormal basis of $\C^m$ consisting of eigenvectors of $A$, and
let   $\{v_1,\dots, v_m\}$ be an orthonormal  basis of $\C^m$ consisting of eigenvectors  of $B$.  Let $Au_i = a_iu_i$ and 
$B v_j = b_jv_j$ for each $i$ and $j$. Then
$$
u_i \cdot \left( \int_0^\infty  (A+xI)^{-1}   \int_0^1 A^{1-s} C B^s \dd s   (B+xI)^{-1} \dd x\right) v_j = 
(u_i\cdot Cv_j) \int_0^1\left[\int_0^\infty \frac{1}{a_i+x}\frac{1}{b_j+x}\dd x \right] a_i^{1-s}b_j^s \dd s \ .\nonumber
$$
Thus it suffices to show that
$$\int_0^1\left[\int_0^\infty \frac{1}{a+x}\frac{1}{b+x}\dd x \right] a^{1-s}b^s \dd s = 1 $$
for all strictly positive numbers $a$ and $b$.  If $a= b$, this is immediately clear. Otherwise, one computes
$$
\int_0^\infty \frac{1}{a+x}\frac{1}{b+x}\dd x =  \frac{1}{a-b}\log(a/b)\ ,
$$
form which the desired result follows directly. 
\end{proof}

The inverse operation will be used frequently in what follows since it provides our ``division by $\rho$'' operation, and so we make a definition:

\begin{definition} Given strictly positive $m\times m$ matrices $A$ and $B$, define the linear transformation $(A,B)\widehat \#$ from the space of $m\times m$ matrices into itself by
\begin{equation}\label{trans3}
(A,B)\widehat \#C = \int_0^\infty  (A+xI)^{-1} C   (B+xI)^{-1} \dd x \ .
\end{equation}
\end{definition}

The following inequalities will be useful:

\begin{lemma}\label{lem:eps-bound}
Let $A$ and $B$ be $m \times m$ matrices satisfying $A, B \geq \eps I$ for some $\eps > 0$. Then, for all $m \times m$ matrices $C$,
\begin{align*}
 {\eps}\tr\big[C^* (A, B) \widehat \# C \big] &\leq  \tr\left[C^* C \right]
 \leq \frac{1}{\eps}\tr\big[C^* (A, B) \# C \big]\;.
\end{align*} 
\end{lemma}

\begin{proof}
Consider the spectral decompositions $A = \sum_i a_i \widetilde u_i$ and $B = \sum_j a_j \widetilde v_j$, where $\widetilde u$ denotes the spectral projection corresponding to the eigenvector $u$. Then we can write $C = \sum_{i,j} c_{ij} \widetilde u_i \widetilde v_j$ for some uniquely determined $c_{ij} \in \C$. Since $a_i, b_j \geq \eps$ by assumption, it follows that
\begin{align*}
 \tr \left[C^* (A, B) \widehat \# C \right]
 & = \sum_{i,j} |c_{ij}|^2 \tr(\widetilde u_i \widetilde v_j)\int_0^\infty
 \frac{1}{a_i+x}\frac{1}{b_j+x}\dd x 
 \\& \leq \frac{1}{\eps} \sum_{i,j} |c_{ij}|^2   \tr(\widetilde u_i \widetilde v_j)
 \\&    =  \frac{1}{\eps} \tr \left[C^*  C \right]\;,
\end{align*}
which proves the first inequality. The second inequality follows from the same argument.
\end{proof}

We also observe:

\begin{lemma}\label{pos}   Given strictly positive $m\times m$ matrices $A$ and $B$,  for all  $m\times m$ matrices $C$,
$$\tr\left[C^* (A,B)\# C\right] \geq 0\ ,$$
and there is equality if and only if $C = 0$. Moreover, for all $m\times m$ matrices $C$ and $D$, 
\begin{align*}
\tr\left[C^* (A,B)\# D\right] = \left({\tr\left[D^* (A,B)\# C\right]}\right)^*\;.
\end{align*}
\end{lemma}

\begin{proof}
It follows from the second inequality in Lemma \ref{lem:eps-bound} that the quantity
$\tr\left[C^* (A,B)\# C\right]$ is non-negative and vanishes if and
only if $C= 0$.

Next, using the fact that for all $m \times m$ matrices $C$, $\tr(C) = \left[\tr(C^*)\right]^*$, and then cyclicity of the trace, 
\begin{align*}
\tr\left[C^* (A,B)\# D\right]
  & =  \int_0^1\tr\left[C^*A^{1-s}D B^s\right]\dd s
\\&= \left(\int_0^1\tr\left[B^s D^* A^{1-s} C  \right]\dd s\right)^*
\\&= \left(\int_0^1\tr\left[ D^* A^{1-s} C B^s \right]\dd s\right)^*
\\& =\left(\tr\left[D^* (A,B)\# C\right]\right)^*\;.
\end{align*}
\end{proof}

Using our new notation, we may rewrite (\ref{basic3}) as
\begin{equation}\label{basic4}
\frac{{\rm d}}{{\rm d}t}S(\rho_t)   = -\tau \left[(\nabla \log \rho_t)^* \cdot  \left(\Gamma(\rho_t),\rho_t \right)\# \nabla \log \rho_t\right]\ .
\end{equation} 
Note that by Lemma~\ref{pos}, the right hand side is strictly negative unless $\rho_t= I$.  We have now achieved a meaningful analog of \eqref{FP1D}
that will lead  us to a meaningful definition of the continuity equation in $\Cl$. Before coming to this, we continue by proving several formulas pertaining to the
inner product implicit in (\ref{basic4}) that will be useful later when we define our 
Riemannian metric on $\Dens$.

\begin{definition}\label{rhonorm}
Let $\rho\in \Dens_+$. For any $A,B\in \Cl$ define the sesquilinear form
$$\langle A,B\rangle_\rho := \tau\left[ A^* \left(\Gamma(\rho),\rho\right)\# B\right]\ ,$$
which by Lemma~\ref{pos} is an inner product on $\Cl$. We define
$$\|A\|_\rho =  \sqrt{\langle A,A\rangle_\rho}$$
to be the corresponding norm. Similarly, for $\AA, \BB \in \Cln$ we define the inner product
\begin{align*}
 \ipo{\AA, \BB}_\rho = \sum_{i=1}^n \ipo{A_i, B_i}_\rho
\end{align*}
and the corresponding norm
\begin{align*}
 \|\AA\|_{\rho} = \sqrt{\sum_{i=1}^n \|A_i\|_\rho^2 }\;.
\end{align*}
\end{definition}

\begin{lemma}[Properties of the inner product $\langle \cdot,\cdot\rangle_\rho$]
\label{ipp}
For any $A,B\in \Cl$,
\begin{eqnarray}
 \langle A,B\rangle_\rho &=& \langle \Gamma(B^*),\Gamma(A^*)\rangle_\rho \ .\label{ipp2}
\end{eqnarray}
Moreover, if $U, V \in \Cln$ are self-adjoint, then
$\ipo{\nabla U, \nabla V}_\rho \in \R$.
\end{lemma}

\begin{proof}
Using cyclicity of the trace and \eqref{eq:gr-alghom}--\eqref{eq:gr-selfadj},
\begin{eqnarray}
\langle A,B\rangle_\rho 
&=& \int_0^1\tau\left[A^*\Gr(\rho)^{s}B\rho^{1-s}\right]\dd s\nonumber\\
&=& \int_0^1 \tau\left[B \rho^{1-s} A^* \Gr(\rho)^s \right]\dd s\nonumber\\
&=& \int_0^1 \tau\left[\Gr(B^*)^* \Gr(\rho)^{1-s} \Gr(A^*) \rho^s \right]\dd s  \nonumber\\
&=&\langle \Gr(B^*),\Gr(A^*)\rangle_\rho\ .\nonumber
\end{eqnarray}
Next, if $U = U^*$ and $V = V^*$ we obtain using \eqref{eq:grad-adj},\eqref{eq:gr-selfadj}, and cyclicity of the trace,
\begin{align*}
 \ipr{\nabla U, \nabla V}
&  = \int_0^1 \tau( (\nabla U)^* \cdot \Gr(\rho)^{1- s} \cdot
  			\nabla V \cdot \rho^{s} ) \dd s
\\&  = \int_0^1 \tau( \Gr(\nabla U) \cdot \Gr(\rho)^{1- s} \cdot
  			\Gr(\nabla V)^* \cdot \rho^{s} ) \dd s
\\&  = \int_0^1 \tau( \nabla U \cdot \rho^{1- s} \cdot
  			(\nabla V)^* \cdot \Gr(\rho)^{s} ) \dd s
\\&  = \int_0^1 \tau(
  			(\nabla V)^* \cdot \Gr(\rho)^{s} \cdot
			 \nabla U \cdot \rho^{1- s}  ) \dd s
\\& = \ipr{\nabla V, \nabla U}\;.
\end{align*}
Since $ \ipr{\nabla U, \nabla V} =  \ipr{\nabla V, \nabla U}^*$ by Lemma \ref{pos}, the claim follows.
\end{proof} 

\subsection{The continuity equation in the Clifford algebra}

Let $\rho(t)$ denote a continuously differentiable curve in $\Dens_+$.  Let us use the notation
$$\dot \rho(t) = \frac{{\rm d}}{{\rm d}t}\rho(t)\ .$$
Then evidently,
$$0 = \tr[\dot \rho(t)] = \ip{I,\dot\rho(t)}\ ,$$
so that $\dot \rho(t)$ is orthogonal to the null space of ${\mathcal N}$. Hence  
$$\dot \rho(t)  =  {\mathcal N}({\mathcal N}^{-1}  \dot \rho(t) )\ .$$
Thus, defining
\begin{equation*}
{\bf A} (t) :=  \nabla({\mathcal N}^{-1}  \dot \rho(t) )\ ,
\end{equation*}
we have
\begin{equation*}
\dot \rho(t) + {\rm div}({\bf A}(t))= 0\ .
\end{equation*}
To write this in the form of  a continuity equation, we use the versions of ``division by $\rho$'' and ``multiplication by $\rho$'' defined in the previous section to define
\begin{equation*}
{\bf V}(t) :=  \left(\Gamma(\rho(t)),\rho(t)\right)\widehat\# {\bf A}(t)\ ,
\end{equation*}
Then by Theorem~\ref{inversion}, we have that
\begin{equation}\label{contequa}
\dot \rho(t) + \dive\left( \left(\Gr(\rho(t)),\rho(t)\right)\#{\bf V}(t)\right)= 0\ .
\end{equation}

\begin{definition}[The continuity equation in the Clifford algebra]Given a  vector field ${\bf V}(t)$ on $\Cl$ depending continuously on $t\in \R$, a continuously differentiable
curve $\rho(t)$ in $\Dens_+$ satisfies the {\em continuity equation for} ${\bf V}(t)$ in case (\ref{contequa}) is satisfied.
\end{definition}

If $\rho(t)$ is a continuously differentiable curve in $\Cl$, then $\dot \rho(t)$ is self-adjoint for each $t$. Considering the definition of the continuity 
equation in the Clifford algebra that we have given, this raises the following question:  {\em For which ${\bf V} \in \Cl^n$ is 
$ {\rm div}\left( \left(\Gamma(\rho(t)),\rho(t)\right)\#{\bf V}\right)$ self-adjoint?}  The following theorem provides an answer that serves our purposes here:

\begin{theorem}\label{selfad}
For $C \in \Cl$ and $\rho\in \Dens_+$ one has 
$$ {\rm div}\left( \left[\Gamma(\rho),\rho\right]\# \nabla (C^*)\right)
  = \Big[{\rm div}\left( \left[\Gamma(\rho),\rho\right]\# \nabla C\right)\Big]^*\;.$$
Consequently, if $C$ is self-adjoint, then
$$ {\rm div}\left( \left[\Gamma(\rho),\rho\right]\# \nabla C\right)$$
is self-adjoint as well.
\end{theorem}

We preface the proof with the following definition and lemma:

\begin{definition}
 We define the antilinear operator $\Gamma_*$ on $\Cl$ by 
 \begin{equation}\label{gamsta0} 
 \Gamma_*(C) := \Gamma(C^*)\ 
 \end{equation}
  for all $C\in \Cl$.
\end{definition}

\begin{lemma}\label{gamstab} 
 For all $\rho\in \Dens_+$,
the operators $[\Gamma(\rho),\rho]\#$ and $\Gamma_*$ commute.
\end{lemma}

\begin{proof}
We compute
\begin{eqnarray}
\Gamma_*\left(\left[\Gamma(\rho),\rho\right]\# C\right)) &=&
\Gamma\left(\int_0^1 \Gamma(\rho)^{1-s} C\rho^s \dd s \right)^*\nonumber\\
&=& \Gamma\left(\int_0^1 \rho^s C^* \Gamma(\rho)^{1-s}  \dd s \right)\nonumber\\
&=&\int_0^1  \Gamma\left(\rho\right)^s  \Gamma\left(C^*\right) \rho^{1-s}  \dd s \nonumber\\
&=&\left[\Gamma(\rho),\rho\right]\#  \Gamma_*(C)\ .\nonumber
\end{eqnarray}
\end{proof}

\begin{proof}[Proof of Theorem \ref{selfad}]
Using \eqref{eq:div-adj} and Lemma \ref{gamstab} we obtain
\begin{align*}
\Big[ \dive \left( \left[\Gamma(\rho),\rho\right]\# \nabla C\right)\Big]^*
  & =  \dive\Gr_*\Big[ \left[\Gamma(\rho),\rho\right]\# \nabla C \Big]
\\& =  \dive\Big[ \left[\Gamma(\rho),\rho\right]\# \Gr_*(\nabla C) \Big]\;.
\end{align*} 
Since \eqref{eq:grad-adj} implies that
$\Gr_*(\nabla C)
   = \nabla (C^*)$,
the result follows.
\end{proof}

\begin{example} Let $\rho\in \Dens$ be given, and define $\rho_t = {\mathcal P}_t\rho_0$. Then by Lemma~\ref{bas},
$$\frac {{\rm d}}{{\rm d}t}\rho(t)
   = -{\mathcal N}\rho(t)
   =   {\rm div}(\nabla \rho(t)) 
   =   {\rm div}\left( \left(\Gamma(\rho(t)),\rho(t)\right)\#\nabla \log \rho(t) \right)\ .$$
Thus, $\rho_t$ satisfies the continuity equation
\begin{equation}\label{contform}
\frac {{\rm d}}{{\rm d}t}\rho(t)  + {\rm div}\left( \left(\Gamma(\rho(t)),\rho(t)\right)\#{\bf V}(t)\right)= 0
\end{equation}
where ${\bf V}(t) = - \nabla \log \rho(t)$. We shall soon see the significance of  the fact that ${\bf V}(t)$ is a gradient. 
\end{example}

We have seen so far that every continuously differentiable  curve $\rho(t)$ in $\Dens_+$ satisfies the  continuity equation for
at least one time dependent vector field ${\bf V(t)}$.  In fact, just as in the classical case, it satisfies the 
 continuity equation for
infinitely many such time dependent vector fields:  Consider any $\rho\in \Dens_+$ and any vector field ${\bf W}\in \Cl^n$. 
Define
\begin{equation}\label{hat1}
\widehat {\bf W} :=  \left(\Gamma(\rho),\rho\right)\widehat \# {\bf W}\ .
\end{equation}
Then by Theorem~\ref{inversion}
$$ {\rm div}({\bf W}) = 0 \quad \iff \quad  {\rm div}\left (  \left(\Gamma(\rho),\rho\right) \# \widehat{\bf W}  \right) = 0\ .$$
We have proved:

\begin{lemma}\label{lem:affine} Let $\rho\in \Dens_+$ and let $\rho(t)$ be a continuously differentiable curve in $\Dens_+$ such that
$\rho(0) = \rho$. Then, for every $t$, the sets of all vector fields ${\bf V}\in \Cl^n$ for which
$$
\dot \rho(t) + \dive\left[ \left(\Gamma(\rho(t)),\rho(t)\right)\#{\bf V}\right] = 0
$$is the affine space consisting of all ${\bf V}\in \Cl^n$ of the form
$${\bf V} = {\bf V}_0 +  \widehat {\bf W}\;, $$
where
$$ {\bf V}_0  :=  \left(\Gamma(\rho(t)),\rho(t)\right)\widehat\#\left[\nabla({\mathcal N}^{-1}  \dot \rho(t) )\right]\ ,$$
and $\widehat {\bf W} :=  \left(\Gamma(\rho(t)),\rho(t)\right)\widehat \# {\bf W}$
where 
$$\dive({\bf W}) =0\ .$$
\end{lemma}

\begin{lemma}\label{decomp} Every ${\bf V}\in \Cl^n$ has a unique decomposition into the sum of a gradient $\nabla U$  and a divergence free
vector field ${\bf Z}$:
$${\bf V} = \nabla U + {\bf Z}\;.$$
In particular, if 
$$\tau\left[ {\bf W}^* \cdot {\bf V}\right]= 0$$
whenever ${\rm div}({\bf W}) = 0$, then ${\bf V}$ is a gradient.
\end{lemma}

\begin{proof} Since   {\rm div}({\bf V})  is orthogonal to the nullspace of ${\mathcal N}$, we may 
define $U:=  -{\mathcal N}^{-1}{\rm div}({\bf V})$. 
Then define
$${\bf Z} := {\bf V} - \nabla U\ .$$
One readily checks that ${\bf V} = \nabla U + {\bf Z}$, and ${\rm div}({\bf Z}) =0$. 

Were the decomposition not unique, there would exist a non-zero vector field that is both a gradient and divergence free. This is impossible since the null space of ${\mathcal N}$ is spanned by $I$.
The final statement now follows easily. 
\end{proof}

The next theorem identifies the ``minimal'' vector field ${\bf V}$  such that a given smooth curve  $\rho(\cdot)$ in $\Dens_+$  satisfies the continuity equation for ${\bf V}$. 
As in the classical case, this identification is the basic step in realizing the $2$-Wasserstein distance as the distance associated to a Riemannian metric.

\begin{theorem}\label{tangent} Let $\rho\in \Dens_+$ and let $\rho(t)$ be a continuously differentiable curve in $\Dens_+$ such that
$\rho(0) = \rho$. Then among all vector fields ${\bf V}\in \Cl^n$ for which
\begin{equation}\label{aff}
\dot \rho(0) + {\rm div}\left[ \left(\Gamma(\rho),\rho\right)\#{\bf V}\right] = 0\ ,
\end{equation}
there is exactly one that is a gradient; i.e., has the from ${\bf V} = \nabla U$ for $U\in \Cl$. Moreover, there exists a self-adjoint element $S \in \Cl$ such that $\nabla U = \nabla S$, and we have
$$\tau\left[(\nabla U)^*\cdot  \left(\Gamma(\rho),\rho\right)\#\nabla U\right] <  \tau\left[{\bf V}^*\cdot  \left(\Gamma(\rho),\rho\right)\#{\bf V}\right] $$
for all other ${\bf V} \in \Cln$  satisfying (\ref{aff}).
\end{theorem}

\begin{proof}  By what we proved in the last subsection, ${\bf V}\mapsto \sqrt{\tau\left[{\bf V}^*\cdot  \left(\Gamma(\rho),\rho\right)\#{\bf V}\right] }$
is an Hilbertian norm on $\Cl^n$. By the Projection Lemma, there is a unique element in the closed convex, in fact, affine,  set
$${\mathcal V} := \left\{ \; {\bf V}\in \Cl^n\ :\  \dot \rho(0) + {\rm div}\left[ \left(\Gamma(\rho),\rho\right)\#{\bf V}\right] = 0\ \right\}$$
of minimal norm. Note that $\mathcal{V}$ is non-empty by Lemma \ref{lem:affine}. Let ${\bf V}_\star$ denote the minimizer. Then by the previous lemma, for each $t\in \R \setminus\{ 0\}$,  and each nonzero ${\bf W}$ such that
${\rm div}({\bf W}) = 0$, 
$$ \tau\left[({\bf V}_\star)^*\cdot  \left(\Gamma(\rho),\rho\right)\#{\bf V}_\star\right] < \tau\left[({\bf V}_\star + t \widehat{\bf W})^*\cdot  \left(\Gamma(\rho),\rho\right)\#({\bf V}_\star + t \widehat{\bf W})\right] \ ,
$$
where $\widehat{\bf W}$ is defined by (\ref{hat1}). Expanding to first order in $t$, and applying Theorem~\ref{inversion}, we conclude
$$\mathfrak{Re}\left( \tau[ {\bf W}^*\cdot {\bf V}_\star ]\right)= 0$$
whenever ${\rm div}({\bf W}) =0$.  Replacing ${\bf W}$ by $i{\bf W}$, we obtain the same conclusion for the imaginary part. By Lemma~\ref{decomp}, this means that ${\bf V}_\star = \nabla U$ for some $U\in \Cl$. 

The proof we have just given shows that in fact any gradient vector field 
in our affine set ${\mathcal V}$  would be a  critical point on the squared norm. But by the strict convexity of the squared norm, there can be only one
critical point. Hence $\nabla U$ is the unique gradient in ${\mathcal V}$.

It remains to show that there exists a self-adjoint element $S \in \Cl$ such that $\nabla U = \nabla S$. For this purpose, we define $S = \frac12 (U + U^*)$ and $A = \frac12(U - U^*)$.
It then suffices to show that $\nabla A = 0$. 
To simplify notation, set $T(C) := \dive\left[ \left(\Gamma(\rho),\rho\right)\#{\nabla C}\right]$.
Using Theorem \ref{selfad} and the fact that $T(U) = -\dot\rho(0)$ is self-adjoint, we infer that
\begin{align*}
 T(S) + T(A) 
 = T(U)
 = T(U^*)
 = T(S^* + A^*)
 = T(S) - T(A)\;,
\end{align*}
hence $\dive\left[ \left(\Gamma(\rho),\rho\right)\#{\nabla A}\right]  = T(A) = 0$. Since we just proved that $\nabla U$ is the unique minimizer in $\mathcal{V}$, we infer that $\nabla A = 0$.
\end{proof}

\subsection{The Riemannian metric}

Theorems \ref{selfad} and \ref{tangent} allow us to identify the tangent space of $\Dens_+$ with the $2^n-1$ dimensional real vector space consisting of all vector fields in $\Cl^n$ which are gradients of self-adjoint elements in $\Cl$:  If $\rho(t)$ is a continuously differentiable curve in $\Dens_+$ with $\rho(0) =\rho\in \Dens_+$, we identify the corresponding tangent vector
with $\nabla U$, where $\nabla U$ is the unique gradient such that (\ref{aff}) is satisfied. We are ready for the central definition:

\begin{definition}
Let $\rho\in \Dens_+$, and let $T_\rho$ denote the tangent space to $\Dens_+$ at $\rho$. The positive definite quadratic from $g_\rho$
on $T_\rho$ is defined by
$$g_\rho(\dot \rho(0), \dot \rho(0)) := \tau\left[(\nabla U)^*\cdot  \left(\Gamma(\rho),\rho\right)\#\nabla U\right] $$
where $\nabla U$ is the unique gradient such that (\ref{aff}) is satisfied. 
\end{definition}

By what we have explained above, this is in fact a Riemannian metric, and indeed is smooth on the manifold $\Dens_+$. 
Let $F$ be a smooth real valued function on $\Dens_+$.  Then the gradient of $F$, denoted ${\rm grad}_\rho (F)$ is the unique vector field
on $\Dens_+$ such that 
whenever $\rho(t)$ is a smooth curve in $\Dens_+$ with $\rho(0) = \rho$,
$$
\frac{{\rm d}}{{\rm d}t} F(\rho(t))\bigg|_{t=0} = g_\rho({\rm grad}_\rho(F),\dot \rho(0))\ .$$

In particular, suppose that $f$ is a real valued, continuously differentiable function on $(0,\infty)$, and $F$ is given by
$$F(\rho) = \tau[f(\rho)]\ .$$
Then by the Spectral Theorem,
$$\frac{{\rm d}}{{\rm d}t} F(\rho(t))\bigg|_{t=0} = \tau[ f'(\rho) \dot\rho(0)]\ .$$
Writing $$\dot \rho(0) +{\rm div}\left( \left(\Gamma(\rho),\rho\right)\#\nabla U\right) = 0\ ,$$
and integrating by parts, this becomes
\begin{eqnarray}
\frac{{\rm d}}{{\rm d}t} F(\rho(t))\bigg|_{t=0} &=& \tau\left [
 \left(\nabla\left(  f'(\rho)\right) \right)^*  \cdot  \left(\Gamma(\rho),\rho\right)\#\nabla U \right]\nonumber\\
&=& g_\rho\left(\nabla\left(  f'(\rho)\right) , \nabla U\right)\ .\nonumber\\
\end{eqnarray}

This computation shows that for a function $F$ on $\Dens_+$ of the form $F(\rho) = \tau[f(\rho)]$, 
\begin{equation}\label{gradform}
{\rm grad}_\rho F = \nabla f'(\rho)\ .
\end{equation}

\begin{definition}\label{gradflow} Given a function $F$ on $\Dens_+$ of the form  $F(\rho) = \tau[f(\rho)]$ where $f$ is smooth on $(0,\infty)$,
the {\em gradient flow equation for $F$ on} $\Dens_+$ is the evolution  equation
$$\frac{{\rm d}}{{\rm d}t} \rho(t) + {\rm div}\left[  \left(\Gamma(\rho(t)),\rho(t)\right)\# \left(-{\rm grad}_{\rho(t)} F\right)  \right] = 0\ ,$$
which by (\ref{gradform}) is equivalent to 
\begin{equation}\label{gradform2}
\frac{{\rm d}}{{\rm d}t} \rho(t) = {\rm div}\left[  \left(\Gamma(\rho(t)),\rho(t)\right)\# \left(\nabla f'(\rho(t))\right)  \right]  \ .
\end{equation}
\end{definition}

We have now completed the work required to prove our first main result:

\begin{theorem}\label{Mehlergf} The flow given by the Fermionic Mehler semigroup is the same as the gradient flow 
$$\frac{{\rm d}}{{\rm d}t} \rho(t) + {\rm div}\left[  \left(\Gamma(\rho(t)),\rho(t)\right)\# \left(-{\rm grad}_{\rho(t)} S\right)  \right] = 0\ ,$$
where $S(\rho)$ is the relative entropy function $\tau [\rho \log \rho]$. 
\end{theorem}

\begin{proof}
Note that $S(\rho) = \tau \left[ f(\rho)\right]$ where  $f(r) = r \log r$. Since $f'(\rho ) = 1 + \log \rho$, we have
$$ {\rm div}\left[  \left(\Gamma(\rho),\rho\right)\# \left({\rm grad}_\rho S\right)  \right]  =
{\rm div}\left[  \left(\Gamma(\rho),\rho\right)\# \left(\nabla \log \rho \right)  \right]\ .$$
Comparison with (\ref{contform}) concludes the proof. 
\end{proof}

This shows once more that if $\rho\in \Dens$, and $\rho(t) := {\mathcal P}_t\rho$, then $S(\rho(t))$ is a strictly decreasing function of $t$
with $\lim_{t\to \infty}S(\rho(t)) =0$. In fact, one can say more: Reversing the steps in the basic computation that led us to 
to the definition of the Riemannian metric, we have 
\begin{eqnarray}\label{diss}
g_\rho(t)(\dot \rho(t), \dot \rho(t))  &=&   \tau \left[ (\nabla \log \rho(t))^* \cdot \left(\Gamma(\rho(t)),\rho(t)\right)\# \nabla \log \rho(t)  \right]\nonumber\\
&=& \tau [(\nabla \log \rho(t) )^* \cdot \nabla \rho(t)]\nonumber\\
&=& -\frac{{\rm d}}{{\rm d t}}S(\rho(t))\ .
\end{eqnarray}
The next lemma quantifies the rate of dissipation of entropy:

\begin{lemma}[Exponential entropy dissipation]\label{entdiss}
Let $\rho(t)$ be any solution of the Fermionic Fokker-Planck equation.  Then
\begin{equation}\label{diss2}
S(\rho(t)) \leq e^{-2t}S(\rho(0))\ .
\end{equation}  
\end{lemma}  
\begin{proof}  This is a direct consequence of Gronwall's inequality, (\ref{diss}), and the modified 
Fermionic Logarithmic Sobolev Inequality
\begin{equation}\label{sfls2}
S(\rho) \leq  \frac12  \tau\left[ \left(\nabla \rho \right)^*\cdot  \nabla  \log \rho\right] \ ,
\end{equation}
for which a simple direct proof is provided in \cite{CL2}. 
\end{proof}

\begin{remark} It is worth noting here that (\ref{sfls2}) can be deduced from the (unmodified)  Fermionic Logarithmic Sobolev Inequality
\begin{equation}\label{sfls}
S(\rho) \leq 2  {\mathcal F}(\rho^{1/2},\rho^{1/2}) \ ,
\end{equation}
that was proved in \cite{CL1}.  To see that (\ref{sfls}) implies (\ref{sfls2}), we recall 
a basic inequality of Gross (see Lemma 1.1 of \cite{G75}), which says that for all $\rho\in \Dens$, and all
$1 < p < \infty$, 
\begin{equation}\label{grosslem}
\tau\left[ \left( \nabla \rho^{p/2}\right)^*\cdot  \nabla \rho^{p/2}\right] \leq \frac{(p/2)^2}{p-1}  \tau\left[ \left(\nabla \rho \right)^*\cdot  \nabla \rho^{p-1}\right] \ .
\end{equation}
Taking the limit $p\to 1$, one obtains the corollary:
\begin{equation}\label{grosslem2}
{\mathcal F}(\rho^{1/2},\rho^{1/2}) 
  = \tau\left[\left( \nabla \rho^{1/2}\right)^*
      \cdot  \nabla \rho^{1/2}\right]
    \leq \frac14  \tau\left[ \left(\nabla  \log \rho\right)^*
 \cdot  \nabla  \rho\right] \ .
\end{equation}
Combining this with (\ref{sfls})
we obtain  (\ref{sfls2}).
\end{remark}

\section{A Talagrand inequality and the diameter of $\Dens$}

\subsection{Arclength, entropy and a Talagrand inequality}

We begin our study of properties of the Riemannian manifold $\Dens_+$ equipped with the metric $g_\rho$ defined in the previous section.

\begin{definition} Let $t\mapsto \rho(t)$ be a continuously differentiable  curve in $\Dens_+$ defined  on $(a,b)$ where
$-\infty \leq a < b \leq +\infty$. Then the {\em arclength} of the curve $\rho(\cdot)$,  ${\rm arclength}[\rho(\cdot))]$, is given by 
 $${\rm arclength}[\rho(\cdot)]:= \int_a^b \sqrt{ g_{\rho(t)}(\dot \rho(t), \dot \rho(t))}\dd t\ .$$
 \end{definition}
 
 Of course, the arc length is independent of the smooth parameterization, and it is always possible to smoothly  reparameterize so that 
 $a =0$ and $b=1$. As usual, this is taken advantage of in the next (standard) definition:

\begin{definition} For $\rho_0,\rho_1\in \Dens_+$, the set  ${\mathcal C}(\rho_0,\rho_1)$ of all {\em couplings of $\rho_0$ and $\rho_1$}  
is the set of all maps $t \mapsto \rho(t)$ from $[0,1]$ to $\Dens_+$ that are smooth on $(0,1)$, continuous on $[0,1]$
 and satisfy $\rho(0)= \rho_0$ and $\rho(1)=  \rho_1$.  The   Riemannian distance between $\rho_0$ and $\rho_1$ is the quantity
 \begin{equation}\label{dist}
 d(\rho_0,\rho_1) =\\ \inf \left\{   \ {\rm arclength}[\rho(\cdot)]\ : \rho(\cdot) \in  {\mathcal C}(\rho_0,\rho_1)\ \right\}\ .
 \end{equation}
In what follows, when we refer to the Riemannian distance on $\Dens_+$, we always mean the distance defined in (\ref{dist}).
\end{definition}

Writing things out more explicitly, for any two $\rho_0,\rho_1\in \Dens_+$,
$$d(\rho_0,\rho_1) = \inf \left\{  \int_0^1 \sqrt{ g_{\rho(t)}(\dot \rho(t), \dot \rho(t))}\dd t\ :\ \rho(\cdot) \in  {\mathcal C}(\rho_0,\rho_1)\ \right\}\ .$$
Yet somewhat more explicitly,
\begin{align*}
d(\rho_0,\rho_1) = \inf \bigg\{  \int_0^1 \|\nabla U(t)\|_{\rho(t)}\dd t\   :\  & \rho(\cdot) \in  {\mathcal C}(\rho_0,\rho_1)\ ,\  \\& \dot\rho(t) + {\rm div}((\Gamma(\rho(t)),\rho(t))\#\nabla U(t)) = 0  \bigg\}\ .
\end{align*}
where
$$\|\nabla U(t)\|_{\rho(t)} := \sqrt{ \tau [ \left(\nabla U(t) \right)^*\left(\Gamma(\rho(t)),\rho(t)\right)\#\nabla U(t)] }\ .$$
This is a direct analog of the Brenier-Benamou formula for the $2$-Wasserstein distance \cite{BB00}, which in turn follows from Otto's Riemannian
interpretation of the $2$-Wasserstein distance \cite{O01}.

Our first goal is to bound the diameter of $\Dens_+$ in the Riemannian metric.   We do this using a Fermionic analog of Talagrand's Gaussian transportation  inequality \cite{Tal96}. The direct connection between logarithmic Sobolev inequalities and Talagrand inequalities was discovered by Otto and Villani \cite{OV00}. Our argument in the present setting uses their ideas, but is also somewhat different.

\begin{theorem}[Talagrand type inequality]\label{tal}  For all $\rho\in \Dens_+$, 
\begin{equation}\label{tal1}
d(\rho,I) \leq  \sqrt{2S(\rho)}\ .
\end{equation}
\end{theorem}

\begin{proof} Given $\rho\in \Dens_+$, define 
$\rho(t) = {\mathcal P}_t\rho$
for $t\in (0,\infty)$.
Since $\lim_{t\to\infty}\rho(t) = I$,  it follows that
$$d(\rho,I) \leq  {\rm arclength}[\rho(\cdot)] =   \int_0^\infty \sqrt{ g_{\rho(t)}(\dot \rho(t), \dot \rho(t))}\dd t\ .$$
By (\ref{diss}),
${\displaystyle g_{\rho(t)}(\dot \rho(t), \dot \rho(t)) = -\frac{{\rm d}}{{\rm d t}}S(\rho(t))}$
so that 
for any $0 \leq t_1 < t_2< \infty$,
\begin{equation}\label{tala}
\int_{t_1}^{t_2} \sqrt{ g_{\rho(t)}(\dot \rho(t), \dot \rho(t))}\dd t \leq \sqrt{t_2-t_1} \sqrt{S(\rho(t_1)) - S(\rho(t_2))}\ .
\end{equation}
Fix any $\epsilon>0$. Define the sequence of times $\{t_k\}$, $k \in \N$, 
$$S(\rho(t_k)) = e^{-k\epsilon}S(\rho)\ .$$
(Since $S(\rho(t))$ is strictly decreasing, $t_k$ is well defined.)
By  Lemma~\ref{entdiss}, for each $k$,

$$t_k - t_{k-1} \leq \frac{\epsilon}{2}\ .$$
Then by (\ref{tala}), with this choice of $\{t_k\}$, 
\begin{align*}
\int_{t_{k-1}}^{t_k} \sqrt{ g_{\rho(t)}{\rho(t)}(\dot \rho(t), \dot \rho(t))}\dd t 
& \ \leq \
 \sqrt{\frac{\epsilon}{2}( e^{-(k-1)\epsilon} - e^{-k\epsilon}){S(\rho)}}
\\& \ = \ \sqrt{\frac{S(\rho)}{2}} e^{-k\epsilon/2}\sqrt{ \epsilon (e^{\epsilon} -1)}\ .
\end{align*}
Since
$$\lim_{\epsilon\to 0}\left(\sum_{k=1}^\infty e^{-k\epsilon/2}\sqrt{ \epsilon (e^{\epsilon} -1)} \right) =
\lim_{\epsilon\to0}\left(\sum_{k=1}^\infty  e^{-k\epsilon/2} \epsilon \right) = \int_0^\infty e^{-x/2}\dd x =2\ ,$$
we obtain the desired bound.  
\end{proof}

\subsection{The diameter of $\Dens_+$}

Since
$$\sup\{ S(\rho)\ :\ \rho\in \Dens_+ \} = \frac{n}{2}\log 2  \ ,$$
we have proved:

\begin{lemma}\label{diamL}
\begin{align}\label{eq:diam}
{\rm diam}(\Dens_+) \leq 2\sqrt{n \log 2}\ .
\end{align}
\end{lemma}

There are other ways to bound the diameter. Given $\rho\in \Dens$, define $\rho(t) = (1-t)\rho + tI$. Then $\rho(\cdot)\in {\mathcal C}(\rho,I)$,
 and $\dot\rho(t) = I - \rho$ for all $t$. As we have seen, $\rho(t)$ satisfies the continuity equation
 $$\frac{{\rm d}}{{\rm d} t} \rho(t) + {\rm div}\left[ 
   (\Gamma(\rho(t)),\rho(t))\#{\bf V}(t)\right] = 0\;,$$
 where 
 $${\bf V}(t) =  \left( \Gamma(\rho(t)),\rho(t)\right)\widehat \# \nabla({\mathcal N}^{-1}(I - \rho))\ .$$
 By the variational characterization of the tangent vector given in Theorem\ref{tangent},
 \begin{eqnarray}
 g_{\rho(t)}(\dot \rho(t),\dot \rho(t)) &\leq& \langle {\bf V}(t) , {\bf V}(t)\rangle_{\rho(t)} \nonumber\\
 &=&  \tau \left[ \left(
 \nabla{\mathcal N}^{-1}(I - \rho) \right)^* \cdot  \left( \Gamma(\rho(t),\rho(t))\widehat \# \nabla{\mathcal N}^{-1}(I - \rho)\right)\right]\;.\nonumber
 \end{eqnarray}
Since $\rho(t) \geq tI$, Lemma \ref{lem:eps-bound} implies that the right-hand side can be bounded from above by
\begin{align*}
\frac1t \tau \left[ \left(
 \nabla{\mathcal N}^{-1}(I - \rho) \right)^* \cdot 
 \nabla{\mathcal N}^{-1}(I - \rho)\right]
 & = \frac1t \tau\left[ (I-\rho){\mathcal N}^{-1}(I-\rho)\right] 
 \\& \leq \frac1t \|I - \rho\|_{L^2(\tau)}^2\ .
\end{align*}
 Thus we have the bound
 $$d(\rho,I) \leq \|I - \rho\|_{L^2(\tau)} \int_0^1 \frac{1}{\sqrt{t}}\dd t =  2\|I - \rho\|_{L^2(\tau)} \ .$$
 This, however, is a cruder bound than the one we obtained using the entropy.

\subsection{Extension of the metric  to $\Dens$}

Our next aim is to show that the distance function $d$ defined on $\Dens_+$, can be continuously extended to $\Dens$. We shall see however in Section \ref{sec:direct} that, even in dimension 1,  the Riemannian metric $g_\rho$ does not extend continuously to the boundary of $\Dens$.

\begin{proposition}\label{prop:extension}
Let $\rho_0, \rho_1 \in \Dens$ and let $\{\rho_0^n\}_n, \{\rho_1^n\}_n$ be sequences in $\Dens_+$ satisfying
\begin{equation}\label{eq:convergence}
 \tau\left[ |\rho_0^n  -\rho_0|^2\right] \to 0\;, \qquad 
 \tau\left[ |\rho_1^n  -\rho_1|^2\right] \to 0\;,
\end{equation}
as $n \to \infty$. Then the sequence $\{  d(\rho_0^n, \rho_1^n) \}_n$ is Cauchy.
\end{proposition}

\begin{proof}
By the triangle inequality, it suffices to show that $d(\rho_0^n, \rho_0^m) \to 0$ as $n, m\to \infty$. 

For this purpose, we fix $\eps \in (0,1)$, set $\bar\rho := (1- \eps)\rho_0 + \eps I$, and take $N \geq 1$ so large that $ \tau\left[ |\rho_0^n  -\rho_0|^2\right] \leq \eps^2$ whenever $n \geq N$. 
Fix $n \geq N$ and consider the linear interpolation $\rho(t) = (1-t) \rho_0^n + t \bar \rho$. Since $\rho(t) \geq t \eps I$ for $t \in [0,1]$, it follows from the definition of $d$ and Lemma \ref{lem:eps-bound} that
\begin{align*}
 d(\rho_0^n, \bar \rho) 
& \leq \int_0^1 
  \sqrt{ \tau [ \dot \rho(t) \cdot \big(\Gamma(\rho(t)),\rho(t)\big)\widehat \#\dot \rho(t)] } \dd t 
\\& \leq \int_0^1  \sqrt{
\frac{ \tau [ |\dot \rho(t)|^2 ] }{t \eps}
    } \dd t \;.
\end{align*}
Since
\begin{align*}
 \tau \big[ |\dot \rho(t)|^2 \big] 
    & = \tau \left[ |\rho_0 - \rho_0^n  + \eps(I - \rho_0)|^2 \right] 
    \\& \leq  2 \tau \left[ |\rho_0 - \rho_0^n |^2 \right] 
         + 2 \eps^2 \tau \left[| I - \rho_0|^2 \right] 
    \\& \leq 2\Big(1 +  \tau \left[| I - \rho_0|^2 \right] \Big) \eps^2\;,     
\end{align*}
we infer that $d(\rho_0^n, \bar \rho) \leq C \sqrt{\eps}$ for some $C$ depending only on $\rho_0$. It follows that $d(\rho_0^n, \rho_0^m) \leq 2C \sqrt{\eps}$ for $n, m \geq N$, which completes the proof.
\end{proof}

In view of this result, the following definition makes sense:

\begin{definition}\label{def:metric}
For $\rho_0, \rho_1 \in \Dens$ we define 
\begin{align*}
 d(\rho_0, \rho_1) := \lim_{n \to \infty}  d(\rho_0^n, \rho_1^n)\;,
\end{align*}
where $\{\rho_0^n\}_n, \{\rho_1^n\}_n$ are arbitrary sequences in $\Dens_+$ satisfying \eqref{eq:convergence}.
\end{definition}

Clearly, for $\rho_0, \rho_1 \in \Dens_+$, this definition is consistent with the one given before.
Note also that $d(\rho_0, \rho_1)$ is finite, since $\Dens_+$ has finite diameter by \eqref{eq:diam}.

We have now proved, in view of Lemma \ref{diamL}:
\begin{theorem}\label{diam2}
\begin{align}\label{eq:diam2}
{\rm diam}(\Dens) \leq 2\sqrt{n \log 2}\ .
\end{align}
\end{theorem}

\section{Characterization of geodesics and geodesic convexity of the entropy}

\subsection{Geodesic equations}

Our next aim is to characterize the geodesics in the Riemannian manifold $\Dens_+$: 
A (constant speed) geodesic is a curve $u : [0,1] \to \Dens$ satisfying 
\begin{align*}
 d(u(s), u(t)) = |t-s| d(u(0), u(1))
\end{align*}
for all $s, t \in [0,1]$. 
Such curves must satisfy a Euler-Lagrange equation that we shall now derive for our Riemannian metric. 
In order to make the argument more transparent, we make a brief detour to a more abstract setting. 
See \eqref{eq:} below for the interpretation of the terms in our Clifford algebra setting.

Let $(V, \ipo{\cdot,\cdot})$ be a finite-dimensional real Hilbert space. Let $W \subset V$ be a linear subspace, fix $z \in V \setminus W$, consider the affine subspace $W_z := z + W$, and let $M \subset W_z$ be a relatively open subset. Let $D : M \to \mathscr{L}(W)$ be a smooth function such that $D(x)$ is self-adjoint and invertible for all $x \in M$. We shall write $C(x) := D(x)^{-1}$. 
Consider the Lagrangian $L : W \times M \to \R$ defined by $L(p, x) = \ipo{C(x)p, p}$
and the associated minimization problem
\begin{align*}
  \inf_{u(\cdot) \in {\mathcal C}^1([0,1], M)} \bigg\{  \int_0^t L(u'(t), u(t)) \dd t \ : 
  				u(0) = u_0\;, \ u(1) = u_1 \bigg\}\;,
\end{align*}
where $u_0, u_1 \in M$ are given boundary values.

Then the Euler-Lagrange equation $\ddt L_p(u',u) - L_x(u',u) = 0$ takes the form 
\begin{align*}
  \ddt C(u(t)) u'(t) - \frac12 \ipo{\partial_x C(u(t)) u'(t), u'(t)}  = 0\;.
\end{align*}
Using the identity $\partial_x C(x) = - C(x)\partial_x D(x) C(x)$ and the substitution $v(t) := C(u(t)) u'(t)$ we infer that the Euler-Lagrange equations are equivalent to the system
\begin{equation}\begin{aligned}\label{eq:EL}
 \left\{ \begin{array}{ll}
  u'(t) - D(u(t))v(t) &= 0\;,\medskip\\
v'(t) + \frac12 \ipo{\partial_x D(u(t)) v(t), v(t)} & = 0\;.\end{array} \right.
\end{aligned}\end{equation}

We shall apply this result to the case where
\begin{equation}\begin{aligned}\label{eq:}
  V =& \,\{A \in \Cl \ :  A \textrm{ self-adjoint} \}\;, \quad
  \ipo{\cdot,\cdot} =  \ip{\cdot,\cdot}\;, \quad
  z = I\;,\\
  W = \Clo := & \,\{ A \in \Cl \ : A \textrm{ self-adjoint}, \ \tau(A) = 0\}\;, \quad
  M = \Dens_+\;,
\end{aligned}\end{equation}
and for any $\rho \in \Dens_+$ the operator $D(\rho): \Clo \to \Clo$ is given by
\begin{align*}
 D(\rho) : U \mapsto
 - \dive\left[ \left(\Gamma(\rho),\rho\right)\#\nabla U\right]\;.
\end{align*}
Note that $D(\rho)$ is invertible for any $\rho \in \Dens_+$, as follows from Theorem \ref{tangent} and the fact that the null space of $\nabla$ consists of multiples of the identity operator. Furthermore, using Lemma \ref{ipp} we infer that $\ip{U, D(\rho) V} \in \R$ for all $U, V \in \Clo$, and
\begin{align*}
\ip{U, D(\rho) V} = \ipo{\nabla U, \nabla V}_\rho 
                  = \ipo{\nabla V, \nabla U}_\rho 
				 = \ip{V, D(\rho) U}\;,
\end{align*}
hence $D(\rho)$ satisfies the assumptions above.
In order to apply \eqref{eq:EL} we use the more general chain rule provided in the Appendix in Propositions \ref{prop:chain-rule} and \ref{prop:identities}  to compute
\begin{align*}
 \ddt\bigg|_{t = 0} (\rho + t \sigma)^\alpha 
   = 
   \int_0^1 \int_0^\alpha 
  \frac{\rho^{\alpha - \beta}}{(1-s)I + s \rho} \sigma
    \frac{\rho^\beta}{(1-s)I + s \rho}  \dd \beta \dd s
\end{align*}
for any $0 < \alpha < 1$,  $\rho \in \Dens_+$, and $\sigma \in \Clo$.
Consequently, for $U \in \Clo$,
\begin{align*}
 & \ddt\bigg|_{t = 0} \ip{D(\rho + t \sigma) U, U}
\\&   =   \ddt\bigg|_{t = 0} \tau \bigg(\int_0^1 (\nabla U)^* \cdot
    \Gr(\rho + t \sigma)^{1-\alpha} \cdot \nabla U \cdot (\rho + t \sigma)^\alpha  \dd \alpha\bigg)
\\&    =    \tau \bigg(\int_0^1  \bigg[ (\nabla U)^* \cdot 
    \bigg(\ddt\bigg|_{t = 0}  \Gr(\rho + t \sigma)^{\alpha} \bigg)\cdot
     \nabla U \cdot \rho^{1-\alpha}
\\& \qquad\qquad    +   (\nabla U)^* \cdot \Gr(\rho)^{1-\alpha} \cdot \nabla U 
 \cdot \bigg(\ddt\bigg|_{t = 0} (\rho + t \sigma)^\alpha\bigg)
     \bigg] \dd \alpha\bigg)
\\&    = 
 \tau \bigg( \int_0^1  \int_0^1 \int_0^\alpha 
  \bigg[(\nabla U )^*\cdot
      \frac{\Gr(\rho)^{\alpha - \beta}}{(1-s)I + s \Gr(\rho)} \Gr(\sigma)
    \frac{\Gr(\rho)^\beta}{(1-s)I + s \Gr(\rho)}  
     \cdot
     \nabla U \cdot \rho^{1-\alpha} 
\\& \qquad\qquad    + (\nabla U )^*\cdot 
    \Gr(\rho)^{1-\alpha} \cdot \nabla U 
 \cdot   \frac{\rho^{\alpha - \beta}}{(1-s)I + s \rho} \sigma
    \frac{\rho^\beta}{(1-s)I + s \rho} 
     \cdot \nabla U 
     \bigg] \dd \beta \dd \alpha \dd s\bigg)\;.
\end{align*}
Using cylicity of the trace and the identities \eqref{eq:gr-alghom} -- \eqref{eq:grad-adj} we obtain
\begin{align*}
  & \ddt\bigg|_{t = 0} \ip{D(\rho + t \sigma) U, U}
   \\& = 
 \tau\bigg( \sigma \cdot \int_0^1  \int_0^1 \int_0^\alpha 
  \bigg[
    \frac{\rho^\beta}{(1-s)I + s \rho}  
     \cdot
     \Gr(\nabla U) \cdot \Gr(\rho)^{1-\alpha} \cdot \Gr_*(\nabla U)  \cdot
     \frac{\rho^{\alpha - \beta}}{(1-s)I + s \rho} 
\\& \qquad\qquad\qquad\qquad    +  
    \frac{\rho^\beta}{(1-s)I + s \rho} 
     \cdot (\nabla U)^* \cdot
      \Gr(\rho)^{1-\alpha} \cdot\nabla U
 \cdot   \frac{\rho^{\alpha - \beta}}{(1-s)I + s \rho}
     \bigg] \dd \beta \dd \alpha \dd s\bigg)
   \\& = 2
 \tau\bigg( \sigma \cdot
  \int_0^1  \int_0^1 \int_0^\alpha 
  \bigg[
	     \frac{\rho^{\alpha - \beta}}{(1-s)I + s \rho}
     \cdot(\nabla U)^* 
 	 \cdot \Gr(\rho)^{1-\alpha} 
	 \cdot \nabla U 
     \cdot \frac{\rho^\beta}{(1-s)I + s \rho} 
     \bigg] \dd \beta \dd \alpha \dd s\bigg)\;.
\end{align*}
Therefore the following definition is natural.

\begin{definition}\label{def:flat}
For $\rho \in \Dens_+$ and $\VV_1, \VV_2 \in \Cln$ we set
\begin{align*}
 \rho \flat (\VV_1, \VV_2)
 & = 2 \int_0^1  \int_0^1 \int_0^\alpha 
  \bigg[
 	     \frac{\rho^{\alpha - \beta}}{(1-s)I + s \rho}
     \cdot \VV_1^* 
 \cdot \Gr(\rho)^{1-\alpha} 
	 \cdot \VV_2 
     \cdot \frac{\rho^\beta}{(1-s)I + s \rho} 
     \bigg] \dd \beta \dd \alpha \dd s\;,
\end{align*}
\end{definition}

\begin{remark}\label{flatrem} If $\rho$, $\Gamma(\rho)$,  $\VV_1$ and $\VV_2$ all commute, it is easy to  explicitly compute the integrals and one finds
\begin{align*}
\rho \flat (\VV_1, \VV_2) =  \VV_1 \cdot  \VV_2
\end{align*} 
in this case.
\end{remark}

With this notation the identity above can be rewritten as  
\begin{align*}
 \ddt\bigg|_{t = 0} \ip{D(\rho + t \sigma) U, U}
	  =  \ip{\sigma,  \rho \flat (\nabla U, \nabla U)}\;,
\end{align*}
and in view of \eqref{eq:EL} we have proved the following result:

\begin{theorem}\label{thm:EL-dens}
The geodesic equations in the Riemannian manifold $\Dens_+$ are given by
 \begin{equation}\begin{aligned}\label{eq:EL-Dens}
 \left\{ \begin{array}{ll}
  \dot\rho(t) +  \dive\left[ \left(\Gamma(\rho(t)),\rho(t)\right)\#\nabla U(t)\right] &= 0\;,\medskip\\
\dot U(t) + \frac12  \rho(t) \flat (\nabla U(t), \nabla U(t)) & = 0\;.\end{array} \right.
\end{aligned}\end{equation}
\end{theorem}

\begin{remark}\label{rem:Rn-compare}
These equations should be compared with the geodesic equations in the Wasserstein space over $\R^n$, which are given by 
\begin{equation}\begin{aligned}\label{eq:geod-Rn}
 \left\{ \begin{array}{ll}
 \partial_t \rho + \nabla \cdot (\rho \nabla U)& = 0\;,\\
 \partial_t \psi + \frac12|\nabla U|^2&  = 0 \;.\end{array} \right.
\end{aligned}\end{equation}
The Fermionic analogue is similar, but note that the second `Hamilton-Jacobi-like' equation in  \eqref{eq:EL-Dens} depends on $\rho$.
However, as explained in Remark~\ref{flatrem}, this dependence is trivial in the presence of sufficient commutativity, in which case 
(\ref{eq:EL-Dens}) reduces to an exact analog of  (\ref{eq:geod-Rn})
\end{remark}

\subsection{The Hessian of the entropy}

Now we are ready to compute the Hessian of the entropy.

\begin{proposition}\label{prop:Hess}
For $\rho \in \Dens_+$ and $U \in \Cl_0$ we have
\begin{equation}\begin{aligned}
\label{eq:Hess}
 \Hess_{\rho} S(\nabla U, \nabla U)
  & = \ip{(\Gr(\rho), \rho) \# \nabla U, \nabla \Num U) }
 - \frac12 \ip{\Num \rho, \rho \flat (\nabla U, \nabla U)}\;.
\end{aligned}\end{equation}
\end{proposition}

\begin{proof}
Let $\rho(t) \in \Dens_+$ and $U(t) \in \Clo$ satisfy the geodesic equations \eqref{eq:EL-Dens}. We shall suppress the variable $t$ in order to improve readability. Using Lemma \ref{bas} we obtain
\begin{align*}
 \ddt S(\rho)
    &  = -\bip{ I + \log \rho , \dive ((\Gr(\rho), \rho)  \#\nabla U) }
  \\&  =  \bip{ \nabla \log \rho ,  (\Gr(\rho), \rho) \# \nabla U }
  \\&  =  \bip{ (\Gr(\rho), \rho)  \#\nabla \log \rho , \nabla U }
  \\&  =  \bip{ \nabla  \rho ,  \nabla U}\;.
\end{align*}
Therefore, using the geodesic equations \eqref{eq:EL-Dens},
\begin{align*}
  \ddtt S(\rho)
      & =  \bip{ \nabla \partial_t \rho ,  \nabla U }
         + \bip{ \nabla  \rho ,  \nabla \partial_t U } 
    \\& =  \bip{  \partial_t \rho ,  \Num U }
         + \bip{ \Num  \rho ,  \partial_t U }
    \\& = -\bip{  \dive ((\Gr(\rho), \rho) \# \nabla U) ,  \Num U }
         + \bip{ \Num  \rho ,  \partial_t U }
    \\& = \bip{   (\Gr(\rho), \rho) \# \nabla U , \nabla \Num U }
         - \frac12 \bip{ \Num  \rho ,  \rho \flat (\nabla U, \nabla U)}\;.
\end{align*}
\end{proof}

\begin{remark}\label{rem:Hess}
The expression (\ref{eq:Hess})  is analogous to the one for the Hessian of the Boltzmann-Shannon entropy ${H}(\rho) = \int_{\R^n} \rho(x) \log \rho(x) \dd x$  in the Wasserstein space over $\R^n$. In that case, 
\begin{equation}\label{rnenthess}
 \Hess_{\rho} {H}(\nabla U, \nabla U)
   = \int_{\R^n}  \Big( 
    \rho \nabla U\cdot  \nabla (-\Delta)U
     - \frac12  (-\Delta \rho )  |\nabla U|^2 \Big) \dd x\;.
\end{equation}
Note that $-\Delta$, like $\Num$, is a positive operator, which is why we have written (\ref{rnenthess}) in terms of $-\Delta$.   In this classical setting, one may simplify (\ref{rnenthess})
using
the identity
$$\frac12 \Delta |\nabla U|^2 = \nabla U\cdot\nabla \Delta  U + \|{\rm Hess}(U)\|^2$$
where  $\|{\rm Hess}(U)\|^2$ denotes the sum of the squares of the entries of the Hessian of $U$. Thus, (\ref{rnenthess}) reduces to
$$ \Hess_{\rho} {H}(\nabla U, \nabla U)  =   \int_ {\R^n}  \|{\rm Hess}(U)\|^2 \rho {\rm d}x\ ,$$
which manifestly displays the positivity of   $\Hess_{\rho} {H}$, and hence the geodesic convexity of the entropy $H$. We lack a simple analog of
$$\Num (\rho \flat (\nabla U, \nabla U))\ ,$$
and thus we lack a simple means to show that the Hessian of $S$ is positive in $\Cl$. In the final section of the paper, we shall show that in fact it is strongly positive  in that 
one even has, for $n=1,2$,
$$
 \Hess_{\rho} S(\nabla U, \nabla U) \geq  \|\nabla U\|^2_\rho\ .
$$
We conjecture that this is true for all $n$. This conjecture is supported by  the close connection between Logarithmic Sobolev Inequalities and entropy, and because the 
Logarithmic Sobolev Inequalities would be a classical consequence if this convexity is true.
\end{remark}

\begin{remark} In addition to the conjecture made in the previous remark, there are many open problems. In the classical case, gradient flows of all sorts of information theoretic
functional of densities lead to physically interesting evolution equations. Whether this is the case in the quantum setting remains to be seen. 

Another open problem concerns the curvature of $\Dens$ in our metric. As Otto has shown,  the $2$-Wasserstein metric on the ``manifold'' of probability measures has non-negative
sectional curvature, which has significant consequences for the general study of gradient flows in the $2$-Wasserstein metric. At present we lack any information on the sectional curvature
in $\Dens$. 
\end{remark}

Our next aim is to prove  Proposition \ref{prop:convex}, which asserts that non-negativity of the Hessian implies that the entropy is convex along geodesics in the metric space $(\Dens,d)$. Since the Riemannian metric degenerates at the boundary of $\Dens$, this implication is not obvious. In order to prove this result, we adapt the Eulerian approach from \cite{OW05,DS08} to our setting.

To carry out the calculations efficiently, we compress our notation at this point. For $\rho \in \Dens_+$, define
\begin{align*}
 \hrho := \int_0^1 \Ga (\rho)^{1-\alpha} \ot \rho^\alpha \dd \alpha 
   \in \Cl \ot \Cl\;.
\end{align*}
With this notation we can write
\begin{align}\label{eq:short-notation}
  (\Ga(\rho), \rho) \# A = \hrho * A\;,
\end{align}
where $*$ denotes the contraction operation 
$$(A\otimes B)*C := ACB\ .$$
Given a curve $t \mapsto \rho(t) \in \Dens_+$ it will be useful to calculate $\ddt \hrho(t)$.

\begin{lemma}\label{lem:dt-rho-hat}
Let $t \mapsto \rho(t) \in \Dens_+$ be a smooth curve. Then 
\begin{align*}
 \ddt\hrho(t) 
  &= \int_0^1 \int_0^1 \int_0^\alpha
     \bigg[\Ga(\rho(t))^{1-\alpha}
      \ot
   \bigg( \frac{\rho(t)^{\alpha - \beta}}{(1-s)I + s \rho(t)} \cdot \dot\rho(t)\cdot \frac{\rho(t)^{\beta}}{(1-s)I + s \rho(t)}
   \bigg)
  \\& \qquad\qquad\qquad+ 
   \bigg(  \frac{\Ga(\rho(t))^{\alpha - \beta}}{(1-s)I + s \Ga(\rho(t))} \cdot \Ga(\dot \rho(t))\cdot \frac{\Ga(\rho(t))^{\beta}}{(1-s)I + s \Ga(\rho(t))}
   \bigg)      \ot    \rho(t)^{1- \alpha}
    \bigg]\dd\beta  \dd \alpha \dd s\;.
\end{align*}
\end{lemma}

\begin{proof}
By the product rule, we have
\begin{align*}
  \ddt \hrho(t) := 
 \int_0^1 \Ga (\rho(t))^{1-\alpha} \ot \ddt\rho(t)^\alpha
  +
    \ddt \big(\Ga(\rho(t))^{1-\alpha}\big) \ot \rho(t)^\alpha 
  \dd \alpha 
  \;.
\end{align*}
Therefore the result follows from the fact that
\begin{align*}
  \ddt \rho(t)^\alpha 
   = \int_0^1 \int_0^\alpha 
  \frac{\rho(t)^{\alpha - \beta}}{(1-s)I + s \rho(t)} \cdot
  		 \dot\rho(t) \cdot  \frac{\rho(t)^{\beta}}{(1-s)I + s \rho(t)} 
		 \dd \beta \dd s\;,
\end{align*}
which is a consequence of Propositions \ref{prop:chain-rule} and \ref{prop:identities}.
\end{proof}

This leads to the following definition.

\begin{definition}\label{def:N-hat}
For $\rho \in \Dens_+$ we define $\hN(\rho) \in \Cl \ot \Cl$ by
\begin{align*}
 \hN(\rho)& = 
 \int_0^1 \int_0^1 \int_0^\alpha
     \bigg[\Ga(\rho)^{1-\alpha}
      \ot
   \bigg( \frac{\rho^{\alpha - \beta}}{(1-s)I + s \rho} \cdot \Num \rho\cdot \frac{\rho^{\beta}}{(1-s)I + s \rho}
   \bigg)
  \\& \quad+ 
   \bigg(  \frac{\Ga(\rho)^{\alpha - \beta}}{(1-s)I + s \Ga(\rho)} \cdot \Ga(\Num\rho)\cdot \frac{\Ga(\rho)^{\beta}}{(1-s)I + s \Ga(\rho)}
   \bigg)      \ot    \rho^{1- \alpha}
    \bigg]\dd\beta  \dd \alpha \dd s\;.
\end{align*}
\end{definition}

Then we have the following result.

\begin{lemma}\label{lem:N-hat}
If $\rho(t) = {\mathcal P}_t \rho$, then 
\begin{align*}
  \ddt  \hrho(t) = - \hN(\rho(t))\;.
\end{align*}
\end{lemma}

\begin{proof}
This is an immediate consequence of Lemma \ref{lem:dt-rho-hat} and Definition \ref{def:N-hat}.
\end{proof}

Now we are ready to state the announced result. Since parts of the argument are very similar to \cite{DS08}, we shall only give a sketch of the proof.

\begin{proposition}\label{prop:convex}
Let $\kappa \in \R$. If $\Hess_{\rho} S(\nabla U, \nabla U) \geq \kappa \ipo{ \nabla U, \nabla U}_{\rho}$ for all $\rho \in \Dens_+$, then for all constant speed geodesics $u :[0,1] \to \Dens$ we have
\begin{align*}
 S(u(t)) \leq (1-t)S(u(0)) + t S(u(1)) 
   				 - \frac{\kappa}{2} t(1-t) d(u(0), u(1))^2\;.
\end{align*}
\end{proposition}

\begin{proof}
For $\rho \in \Dens$ and $U \in \Cl_0$ we set 
\begin{align*}
 \cA(\rho, U)
   =  &\  \| \nabla U \|_\rho^2  =
    \ip{\hrho  \contr\,\nabla U\ ,\, \nabla U}\;,\\
 \cB(\rho, U)
    = &\ \Hess_\rho S(\nabla U, \nabla U) 
= 
    \ip{\hrho  \contr\,\nabla U\ ,\, \nabla \Num U}
- \frac12
  \ip{\hN\rho\, \contr\, \nabla U \ , \nabla U }\;.
\end{align*}

Let $\{\rho^s\}_{s \in [0,1]}$ be a smooth curve in $\Dens_+$ and set $\rho_t^s := {\mathcal P}_{st} \rho^s$ for $t \geq 0$. Let $\{U_t^s\}_{s \in [0,1]}$ be a smooth curve in $\Cl_0$ satisfying the
  continuity equation
\begin{align*}
  \partial_s \rho_t^s + \dive ( \hrho_t^s \contr \nabla U_t^s) =
  0\;, \qquad s \in [0,1]\;.
\end{align*}
We claim that the identity
\begin{align*}
 \frac12 \partial_t \cA(\rho_t^s, U_t^s)
 	 + \partial_s \Ent(\rho_t^s) = - s \cB(\rho_t^s, U_t^s)
\end{align*}
holds for every $s \in [0,1]$ and $t \geq 0$.
Once this is proved, the result follows from the argument in \cite[Section 3]{DS08} (see also \cite[Theorem 4.4]{EM11} where this program has been carried out in a discrete setting). 

To prove the claim, we calculate
\begin{equation}\begin{aligned}\label{eq:deriv-ent}
\partial_s \Ent(\rho_t^s)
  & = \bip{I + \log \rho_t^s \ ,\,  \partial_s \rho_t^s}
\\& = - \bip{I + \log \rho_t^s \ ,\,
		 \dive ( \hrho_t^s \contr \nabla U)}
\\& =  \bip{ \nabla \log \rho_t^s \ ,\,
   				\hrho_t^s \contr \nabla U_t^s}
\\& =  \bip{ \nabla \rho_t^s\ ,\, \nabla U_t^s}
\\& =   \bip{U_t^s \ ,\,\Num \rho_t^s}\;.
\end{aligned}\end{equation}
Furthermore, 
\begin{align*}
 \frac12 \partial_t \cA(\rho_t^s, U_t^s)
&  =  \bip{ \hrho_t^s \contr \partial_t\nabla U_t^s\ ,\, 
    			\nabla U_t^s}
 + \frac12
			 \bip{ \partial_t\hrho_t^s \contr \nabla U_t^s\ ,\,
    			\nabla U_t^s }
 \\& =: I_1 + I_2\;.			
\end{align*}
In order to simplify $I_1$ we claim that 
\begin{align} 
 \label{eq:partial-ts}
  -\dive \big( (\partial_t \hrho_t^s) \contr \nabla U_t^s \big) 
  -\dive \big( \hrho_t^s \contr \partial_t \nabla U_t^s \big)
   & =  s \Num\big(\dive(\hrho_t^s \contr\nabla U_t^s)\big)
    - \Num \rho_t^s\;,\\
  \label{eq:partial-t}
  \partial_t \hrho_t^s & = -s \hN \rho_t^s\;.
\end{align}
To show \eqref{eq:partial-ts}, note that the left-hand side equals
$\partial_t \partial_s \rho_t^s$, while the right-hand side equals
$\partial_s \partial_t \rho_t^s$. The identity \eqref{eq:partial-t}
follows from Lemma \ref{lem:N-hat}.

Integrating by parts repeatedly and using \eqref{eq:deriv-ent},
\eqref{eq:partial-ts} and \eqref{eq:partial-t}, we obtain
\begin{align*}
  I_1 
  & = - \bip{ U_t^s\ ,\,
    		\dive (\hrho_t^s \contr \partial_t\nabla U_t^s)}
\\  & =  - \bip{ U_t^s\ ,\, \Num \rho_t^s}
	    + s \bip{ U_t^s\ ,\,
	 \Num\big(\dive (\hrho_t^s \contr \nabla U_t^s)\big)}
\\& \qquad  +  \bip{ U_t^s\ ,\,
  		   \dive
		    \big(  (\partial_t \hrho_t^s) \contr \nabla U_t^s \big) }
\\  & =  - \partial_s \Ent(\rho_t^s)
		 -  s  \bip{ \hrho_t^s \contr \nabla U_t^s\ ,\,
   			\nabla \Num U_t^s}  
	 + s  \bip{\hN \rho_t^s \contr \nabla U_t^s\ ,\,
	    \nabla U_t^s}\;.
\end{align*}
Taking into account that
\begin{align*}
 I_2 = -\frac{s}2 
				\bip{\hN \rho_t^s \contr \nabla U_t^s\ ,\,
    					\nabla U_t^s }\;,
\end{align*}
the result follows by summing the expressions for $I_1$ and $I_2$.
\end{proof}

\section{Direct verification of the $1$-convexity of the entropy}
\label{sec:direct}

Our results in this section support the conjecture made in Remark~\ref{rem:Hess}. We shall show that for $n=1,2$, the entropy is $1$-convex along geodesics in the metric space $(\Dens, d)$. This notion of convexity may be seen as a Fermionic analog of McCann's displacement convexity \cite{McC97}, which corresponds to convexity along geodesics in the $2$-Wasserstein space of probability measures.

\subsection{The 1-dimensional case}

In this section we shall perform some explicit computations in the Riemannian manifold $\Dens_+$ in the special case where the Clifford algebra is 1-dimensional.

In this case the Clifford algebra is commutative and consists of all elements of the form $X = x I + y Q$ with $x,y \in \C$ and $Q = Q_1$. The set of probability densities is given by 
\begin{align*}
 \Dens = \{ \rho_y = I + y Q \ : \  -1 \leq y \leq 1 \}\;,
\end{align*}
and $\rho_y$ belongs to $\Dens_+$ if and only if $-1 < y < 1$. Our aim is to calculate the distance $d(\rho_{y_0}, \rho_{y_1})$ explicitly. For this purpose, we observe that for $p > 0$,
\begin{align*}
  (\rho_y)^p & = (1-y)^p \frac{1-Q}{2}
               + (1+y)^p \frac{1+Q}{2}
           \\& = \frac{(1+y)^p  + (1-y)^p }{2} I
                  + \frac{(1+y)^p  - (1-y)^p }{2} Q
           \\& =: c_p(y) I + d_p(y) Q        \;.
\end{align*}
Note also that $(\Gamma(\rho_y))^p =  c_p(-y) I + d_p(-y) Q$. Therefore, if $U = u_0 I + u Q$ and $\VV = \nabla U = u I$, then  
\begin{align*}
 (\Gr(\rho_y), \rho_y)\# \VV 
  & =   \int_0^1 \Big(  c_{1-p}(-y) I + d_{1-p}(-y) Q \Big)  
   				\Big(  c_{p}(y) I + d_{p}(y) Q \Big) \dd p \cdot u I
 \\&  = \int_0^1 (1-y)^{1-p} y^p \dd p \cdot  u I
 \\&  = \frac{ y}{\arctanh(y)} u I\;.
\end{align*}
We infer that 
\begin{align*}
 \dive (\Gr(\rho_y), \rho_y)\# \VV 
   = -\frac{  y}{\arctanh(y)}u Q\;,
\end{align*}
hence, if $\rho(t) = \rho_{y(t)}$ and $\nabla U(t) = u(t) I$, then 
the continuity equation
\begin{align*}
 \ddt \rho(t)  
 + \dive \left( \left(\Gamma(\rho(t)),\rho(t)\right)\#{\nabla U}(t)\right)
 = 0
\end{align*}
is equivalent to 
\begin{align}\label{eq:CE-dim1}
 \dot y(t) -  \frac{y(t)}{\arctanh(y(t))} u(t) = 0\;.
\end{align}
Furthermore, since
\begin{align*}
 \| \nabla U(t) \|_{\rho(t)}^2
&   = - \bip{U(t),\dive \left( \left(\Gamma(\rho(t)), \rho(t)\right)\#{\nabla U}(t)\right)}
 \\&  = \frac{  y(t)}{\arctanh(y(t))}u^2(t)\;,
\end{align*}
we obtain  for $y_0, y_1 \in (-1,1)$,
\begin{align}\label{eq:min-pblm}
  d( \rho_{y_0}, \rho_{y_1})^2
    = \inf_{y,u}  
      \bigg\{
      \int_0^1 \frac{ y(t)}{\arctanh(y(t))} u^2(t) \dd t 
      \bigg\}\;,
\end{align}
where the infimum runs over all smooth functions $y : [0,1] \to (-1,1)$ and $u : [0,1] \to \R$ satisfying \eqref{eq:CE-dim1}
with boundary conditions $y(0) = y_0$ and $y(1) = y_1$. 

This metric coincides with the Riemannian metric studied in \cite[Section 2]{Ma11} in the special case of a Markov chain $K$ on a two-point space $\mathcal{X} = \{ a,b \}$ with transition probabilities $K(a,a) = K(a,b) = K(b,a) = K(b,b) = \frac12$. The minimization problem in \eqref{eq:min-pblm} can be solved explicitly (see \cite[Theorem 2.4]{Ma11}), and for $-1 < y_0 < y_1 < 1$ one obtains 
\begin{align} \label{eq:explicit}
 d(\rho_{y_0}, \rho_{y})
   = \int_{y_0}^{y_1} \sqrt{ \frac{\arctanh(y)}{y} } \dd y\;.
\end{align}

Note that the function $y \mapsto  \sqrt{ \frac{y}{\arctanh(y)} }$ diverges as $y \to \pm 1$; this corresponds to the fact that the Riemannian metric degenerates at the boundary of $\Dens_+$. However, the improper integral in \eqref{eq:explicit} does converge if $y_0 = -1$ or $y_1 =1$, which can be seen directly and can also be inferred from Theorem \ref{tal} and Proposition \ref{prop:extension}. 

Let $-1 < y_0 < y_1 < 1$. It has been shown in \cite[Proposition 2.7]{Ma11} that the geodesic equation for a curve $[0,1] \ni t \mapsto \rho_{y(t)} \in \Dens_+$ connecting $\rho_{y_0}$ and $\rho_{y_1}$, is given by
\begin{align}\label{eq:geod-1}
  y'(t) = d(\rho_{y_0}, \rho_{y_1}) \sqrt{\frac{\arctanh(y(t))}{y(t)}}\;.
\end{align}
Moreover, if $y(t)$ satisfies \eqref{eq:geod-1}, then the second derivative of the entropy is given by
\begin{align*}
\ddtt S(\rho_{y(t)}) = \frac{d(\rho_{y_0}, \rho_{y_1})^2}{2} 
			\bigg( 1 + \frac{1}{1-y(t)^2} \frac{y(t)}{\arctanh(y(t))}\bigg)\;,
\end{align*}
which implies that
\begin{align*}
  S(\rho_{y(t)}) \leq (1-t)S(\rho_{y_0}) + t S(\rho_{y_1}) 
   				 - \frac{1}{2} t(1-t) d(\rho_{y_0},\rho_{y_1})^2\;,
\end{align*}
thus $S$ is 1-convex along geodesics. We refer to \cite[Section 2]{Ma11} for more details.

\subsection{The 2-dimensional case}

As in the 1-dimensional case, our goal is to obtain an explicit formula for the Hessian of the entropy $S$ and to show that it is bounded from below. First we shall describe the set of probability densities. For this purpose, it will be useful to introduce the notation 
\begin{align*}
 \rho_\rr = I + x Q_1 + y Q_2 + i z Q_1 Q_2
\end{align*}
for $\rr = (x,y,z) \in \C^3$.

With this notation, the set of probability densities can be characterized as follows.

\begin{lemma}\label{lem:densities}
We have
\begin{align*}
 \Dens = \{ \rho_\rr \in \Cl \ : \ \rr = (x,y,z) \in \bar B  \}\;,
\end{align*}
where $\bar B$ denotes the closure of the unit ball in $\R^3$. Moreover, $\rho_\rr$ belongs to $\Dens_+$ if and only if $\rr$ belongs to the open unit ball $B$.
\end{lemma}

\begin{proof}
Let $X \in \Cl$ be of the form
\begin{align*}
 X = w +  x Q_1 + y Q_2 +  i z Q_1 Q_2
\end{align*} 
for some $w, x, y, z \in \C$. Clearly, $X$ is self-adjoint if and only if $w, x, y, z \in \R$. In this case, one readily checks that the spectrum of $X$ consists of the two elements
\begin{align*}
 w \pm \sqrt{x^2 + y^2 + z^2}\;,
\end{align*}
both of which have multiplicity $2$. This implies both assertions, taking into account that $\tau(X) = w$.
\end{proof}

In order to obtain explicit formulas for expressions of the form $(\Gr(\rho), \rho) \# \nabla U$ with $\rho \in \Dens_+$ and $U \in \Clo$, one needs to evaluate fractional powers of $\rho$. The following result describes the functional calculus of elements in $\Dens$.

\begin{lemma}\label{lem:funct-calc}
For $\rr \in B \setminus \{ 0 \}$ and $f : [0,2] \to \R$ we have
\begin{align*}
 f(\rho_\rr) = 
    \frac{f(1 - |\rr|)}{2} \rho_{-\nn}
  + \frac{f(1 + |\rr|)}{2} \rho_{\nn}\;, 
\end{align*}
where $\nn = \frac{1}{|\rr|}\rr$.
\end{lemma}

\begin{proof}
One easily checks that an element 
\begin{align*}
 X = w +  x Q_1 + y Q_2 +  i z Q_1 Q_2
\end{align*}
is a projection if and only if $X = \frac12 \rho_\rr$ for some $\rr \in \partial B$, where $\partial B$ denotes the unit sphere in $\R^3$.
Furthermore, two projections $X^{(1)} = \frac12 \rho_{\rr^{(1)}}$ and $X^{(2)} = \frac12 \rho_{\rr^{(2)}}$ are mutually orthogonal if and only if $\rr^{(1)} = - \rr^{(2)}$. As a consequence, the spectral decomposition of $\rho_{\rr}$ with $\rr \in B$ is given by 
\begin{align*}
 \rr =  (1 - |\rr|) P_{(-)} + (1+|\rr|) P_{(+)}
\end{align*}
where $P_{(\pm)} = \frac12 \rho_{\pm \nn}$ and $\nn = \frac{1}{|\rr|}\rr$. This implies the desired result.
\end{proof}

In the following computations, 
an important role will be played by the logarithmic mean $\mu(x,y)$, which is defined for $x , y \geq 0$ by
\begin{align*}
 \mu(x,y) = \int_0^1 x^{1-\alpha} y^\alpha \dd \alpha\;.
\end{align*}

Let us fix the notation that shall be used throughout the remainder of this section. We consider a fixed element $\rho \in \Dens_+$ of the form 
\begin{align*}
\rho = I + x Q_1 + y Q_2 + iz Q_1 Q_2
\end{align*}
 for some $x, y, z \in \R$ satisfying
\begin{align*}
r := \sqrt{x^2 + y^2 + z^2} \in (0,1)\;.
\end{align*} 
It will be useful to introduce the quantities
\begin{align*}
\theta := \mu(1-r,1+r) =  \frac{r}{\arctanh(r)}\;.
\end{align*}
Furthermore, we set $a := x/r$ , $b: = y/r$, $c: = z/r$, and 
\begin{align*}
 \mm = (-a,-b, c)\;, \qquad
 \nn = ( a, b, c)\;.
\end{align*}

\begin{lemma}\label{lem:hat-explicit}
Let $\rho \in \Dens_+$ and $U \in \Cl$.
With the notation from above we have
\begin{align*}
 (\Gr(\rho), \rho)\# U 
   = \frac14 \sum_{\eps_1, \eps_2 \in \{-1,1\}}
      \mu(1+\eps_1 r , 1+ \eps_2 r)
      \rho_{\eps_1 \mm} U \rho_{\eps_2 \nn}\;.
\end{align*}
\end{lemma}

\begin{proof}
This readily follows from Lemma \ref{lem:funct-calc}.
\end{proof}

With the help of this lemma, it is straightforward to obtain the following identities.

\begin{lemma}\label{lem:tedious}
For $\rho \in \Dens_+$ the following identities hold:
\begin{align*}
 (\Gr(\rho), \rho)\# I 
&    = 
     \Big(\theta(a^2 + b^2) + c^2 \Big) I 
 + i  b c(1-\theta) Q_1 -i a c(1-\theta) Q_2 + ic r Q_1 Q_2\;,\\
 (\Gr(\rho), \rho)\# (iQ_1)
&    = 
 b c(1-\theta) I 
+ i(\theta(a^2 + c^2) + b^2) Q_1 -i a b(1-\theta) Q_2 + ib r Q_1 Q_2\;,\\
 (\Gr(\rho), \rho)\# (iQ_2)
& = 
     -a c(1-\theta) I 
    -i a b(1-\theta) Q_1 + i(\theta(b^2 + c^2) + a^2) Q_2
      - ia r Q_1 Q_2\;.
\end{align*}
\end{lemma}

\begin{proof}
This follows from a direct computation based on Lemma \ref{lem:hat-explicit}.
\end{proof}

Using this lemma we can obtain an explicit expression for the Riemannian metric. With the notation from above we obtain the following result.

\begin{lemma}\label{lem:metric2d}
Let $\rho \in \Dens_+$ and let $U \in \Cl$ be of the form $U = u Q_1+ v Q_2 + i w Q_1 Q_2$ for some $u, v, w \in \R$. 
Then
\begin{align*}
 \ipo{ \nabla U, \nabla U }_{\rho}
   = \uu^T M(\rho) \uu\;,
\end{align*}   
where the right-hand side is a matrix-product with $\uu^T = (u,v,w)$ and
\begin{align*}
M(\rho) = \left(\begin{array}{ccc}
  \theta(a^2 + b^2) + c^2 & 0 & (\theta -1)ac \\
  0 &   \theta(a^2 + b^2) + c^2 &  (\theta -1)bc \\
  (\theta -1)ac &  (\theta -1)bc &   a^2 + b^2 + \theta(1 + c^2)
  \end{array}\right) \;.
\end{align*}
\end{lemma}

By a similar calculation one can compute the fist term appearing in the expression \eqref{eq:Hess} for the Hessian of the entropy $S$ at $\rho$.

\begin{lemma}\label{lem:Hess-1}
Let $\rho \in \Dens_+$ be as in Lemma \ref{lem:hat-explicit} and 
let $U \in \Cl$ be of the form $U = u Q_1+ v Q_2 + i w Q_1 Q_2$ for some $u, v, w \in \R$. 
Then,
\begin{align*}
 \ip{(\Gr(\rho), \rho) \# \nabla U, \nabla \Num U) }
   = \uu^T N_1(\rho) \uu\;,
\end{align*}   
where the right-hand side is a matrix-product with $\uu^T = (u,v,w)$ and
\begin{align*}
N_1(\rho) = \left(\begin{array}{ccc}
  \theta(a^2 + b^2) + c^2 & 0 &  \frac32 (\theta -1)ac \\
  0 &   \theta(a^2 + b^2) + c^2 &   \frac32 (\theta -1)bc \\
  \frac32(\theta -1)ac &   \frac32 (\theta -1)bc &  2( a^2 + b^2 + \theta(1 + c^2))
  \end{array}\right) \;.
\end{align*}
\end{lemma}

With some additional work the second part in the expression \eqref{eq:Hess} for the Hessian  can be characterized as well. It turns out that the following generalization of the logarithmic mean plays a role. For $x , y, z \geq 0$ we set
\begin{align*}
 \mu(x,y,z) = 2 \int_0^1 \int_0^\alpha  x^{1-\alpha} y^{\alpha - \beta} z^\beta \dd \beta \dd \alpha\;.
\end{align*}

The following result gives an explicit expression for $(\Gr(\rho), \rho)\# U$.

\begin{lemma}\label{lem:flat-explicit}
For $\rho \in \Dens_+$ and $\VV_1, \VV_2 \in \Cl^2$ we have 
\begin{align*}
 \rho \flat (\VV_1, \VV_2)
   = \frac18 \sum_{\eps_1, \eps_2, \eps_3 \in \{-1,1\}}
    \frac{  \mu(1+\eps_1 r, 1+ \eps_2 r,  1+ \eps_3 r) }
	    {  \mu(1+\eps_1 r, 1+ \eps_3 r) }
      \rho_{\eps_1 \nn} \VV_1^* \rho_{\eps_2 \mm} \VV_2  \rho_{\eps_3 \nn} 
\end{align*}
\end{lemma}

\begin{proof}
This follows using Lemma \ref{lem:funct-calc} and the definition of $\rho \flat (\VV_1, \VV_2)$.
\end{proof}

The identity from the previous lemma allows us to obtain an explicit expression for the second term in  the Hessian of the entropy:

\begin{lemma}\label{lem:Hess-2}
Let $\rho \in \Dens_+$ and let $U \in \Cl$ be of the form $U = u Q_1+ v Q_2 + i w Q_1 Q_2$ for some $u, v, w \in \R$. Furthermore, we 
set  $\xi = \mu(1-r,1-r,1+r)$ and $\eta = \mu(1-r,1+r,1+r)$, and we consider the quantities
\begin{align*}
 \Gamma = \frac{r}{4}\frac{\eta -\xi}{\theta}\;, \qquad
 \Delta = \frac{r}{4} \left( \frac{\xi}{1-r} - \frac{\eta}{1+r}  \right)\;.
\end{align*}
Then, 
\begin{align*}
  - \frac12 \ip{\Num \rho, \rho \flat (\nabla U, \nabla U)}
   = \uu^T N_2(\rho) \uu
\end{align*}
where the right-hand side is a matrix-product with $\uu^T = (u,v,w)$ and
\begin{align*}
N_2(\rho) = \left(\begin{array}{ccc}
  A & 0 & aC \\
  0 & A & bC \\
  aC &  bC & B
  \end{array}\right) \;,
\end{align*}
with 
\begin{align*}
 A &= (1 - c^2) (  (1+c^2) \Delta -2c^2 \Gamma)\;,\\
 B &= (1+c^2)^2 \Delta + 2c^2(1-c^2)\Gamma\;,\\
 C &= c((1+c^2) \Delta + (1-2c^2)\Gamma)\;.
\end{align*}
\end{lemma}

Now that we have obtained explicit formulas for the metric and the Hessian, we are ready to prove the following result.

\begin{theorem}\label{thm:2D-convex}
For all $\rho \in \Dens_+$ and all selfadjoint elements $U \in \Cl$ we have
\begin{align*} 
 \Hess_\rho S(\nabla U, \nabla U) \geq \| \nabla U\|_\rho^2\;.
\end{align*} 
\end{theorem}

\begin{proof}
It follows directly from Lemmas \ref{lem:metric2d}, \ref{lem:Hess-1}, and \ref{lem:Hess-2} that for $\rho$ and $U$ as in these lemmas, 
\begin{align*}
 \Hess_\rho S(\nabla U, \nabla U) - \| \nabla U\|_\rho^2
   = \uu^T P(\rho) \uu \;,
\end{align*}
where
\begin{align*}
 P(\rho)  =  N_1(\rho) +  N_2(\rho) - M(\rho)  
   = 
   \left(\begin{array}{ccc}
  {\tilde A} & 0 & a{\tilde C} \\
  0 & {\tilde A} & b{\tilde C}\\
  a{\tilde C} & b{\tilde C} & {\tilde B}
  \end{array}\right) \;,
\end{align*}   
where  
\begin{align*}
 {\tilde A} &= A =  (1 - c^2) (  (1+c^2) \Delta -2c^2 \Gamma)\;,\\
 {\tilde B} &= (1-c^2) + \theta(1+c^2) + (1+c^2)^2 \Delta + 2c^2(1-c^2)\Gamma\;,\\
  \tilde C &=  c\Big( \frac12 (\theta -1) + (1+c^2) \Delta + (1-2c^2)\Gamma \Big)\;.
\end{align*}
An elementary computation shows that a matrix of this form is positive definite
if and only if ${\tilde A} \geq 0$, ${\tilde B} \geq 0$, and
\begin{align}\label{eq:xw}
 {\tilde A}{\tilde B} \geq {\tilde C}^2 (a^2 + b^2 ) \;.
\end{align}
The proof of these inequalities relies on the following one-dimensional inequalities, which shall be proved in Proposition \ref{prop:essential} below:
\begin{align}\label{eq:easy}
0\leq 2 \Gamma &\leq \Delta\;,\\
\label{eq:difficult}
 (1-\theta)^2 &\leq 4 \Delta\;.
\end{align}
In fact, the non-negativity of ${\tilde A}$ and ${\tilde B}$ follows immediately from  \eqref{eq:easy}. 
In order to prove \eqref{eq:xw} we write
\begin{align*}
 {\tilde A}{\tilde B} - \tilde C^2(a^2 + b^2)
  & = (1-c^2) (\mathcal{A} + \mathcal{B})
\end{align*}
where 
\begin{align*}
\mathcal{A} &= (1+c^2)^2\Delta^2 - c^2 \Gamma^2 - 2c^2(1+c^2)\Gamma \Delta\;,\\
\mathcal{B} &=  (1+c^2)(1+\theta)\Delta - c^2(1+3\theta)\Gamma - \frac14c^2(1-\theta)^2\;.
\end{align*}
Using \eqref{eq:easy} we infer that
\begin{align*}
\mathcal{A} & = (c^4 + c^2)\Delta(\Delta-2\Gamma)
			 + (1+\frac34c^2)\Delta^2
			 + c^2( \frac14 \Delta^2 - \Gamma^2)
	     \\ &\geq 0\;.
\end{align*}
Furthermore, taking into account that $0 \leq \theta \leq 1$, using \eqref{eq:easy} once more, and finally \eqref{eq:difficult}, we obtain
\begin{align*}
\mathcal{B} &=
(1+\theta)\Delta - \frac14c^2(1-\theta)^2
+ c^2 ( (1+\theta)\Delta  - (1+3\theta)\Gamma  )
\\& \geq \Delta -  \frac14 (1-\theta)^2
 + c^2  (1+\theta) (\Delta  - 2 \Gamma  )
\\& \geq \Delta -  \frac14 (1-\theta)^2
\\& \geq 0\;,
\end{align*}
which completes the proof.
\end{proof}

The following one-dimensional inequalities were essential in the proof of Theorem \ref{thm:2D-convex}.

\begin{proposition}\label{prop:essential}
For $-1\leq r \leq 1$ we set $\theta = \mu(1-r,1+r)$ and 
\begin{align*}
   \xi = \mu(1-r,1-r,1+r)\;,\qquad
  \eta = \mu(1-r,1+r,1+r)\;.
\end{align*}
Then the quantities
\begin{align*}
 \Gamma = \frac{r}{4}\frac{\eta -\xi}{\theta} \qquad \text{and}\qquad
 \Delta = \frac{r}{4} \left( \frac{\xi}{1-r} - \frac{\eta}{1+r}  \right)
\end{align*}
satisfy the following inequalities:
\begin{align}\label{eq:easy-lem}
0\leq 2 \Gamma &\leq \Delta\;,\\
\label{eq:difficult-lem}
 (1-\theta)^2 &\leq 2 \Delta\;.
\end{align}
\end{proposition}

\begin{proof}
The first inequality from \eqref{eq:easy-lem} is clear from the monotonicity of $\mu$.
It follows from the 1-homogeneity of $\mu$ that the second inequality in \eqref{eq:easy-lem} can be reformulated as
\begin{align}\label{eq:equiv1}
      \Big(1 + \frac{2(1+r)}{\theta}\Big) \mu(1,1,c^{-1})
 \leq \Big(1 + \frac{2(1-r)}{\theta}\Big) \mu(1,1,c)\;,
\end{align}
where $c = \frac{1+r}{1-r}$. Using the identity
\begin{align*}
 \frac{\mu(1,1,c)}{\mu(1,1,c^{-1})}
   = \frac{\frac{\theta}{1-r} - 1}{1 - \frac{\theta}{1+r}}
\end{align*}
it follows that \eqref{eq:equiv1} is equivalent to 
\begin{align}\label{eq:log-geom}
G \leq \frac{\theta}{\sqrt{2-\theta}}\;,
\end{align}
where $G = \sqrt{1-r^2}$ is the geometric mean of $1-r$ and $1+r$.
Since \eqref{eq:log-geom} is readily checked, we obtain \eqref{eq:easy-lem}.

In order to prove \eqref{eq:difficult-lem}, we use the identity
\begin{align*}
 \Delta = \frac{\theta}{4}\bigg( \frac{2\theta}{1-r^2} - 2 \bigg)\;.
\end{align*}
Therefore the inequality \eqref{eq:difficult-lem} is equivalent to 
\begin{align*}
  (1-\theta)^2 \leq {\theta}\bigg( \frac{\theta}{1-r^2} - 1 \bigg)\;.
\end{align*}
In view of the geometric-logarithmic mean inequality $\sqrt{1-r^2} \leq \theta$, it suffices to show that
\begin{align*}
  \theta(1-\theta)^2 \leq  {\theta} - 1 + r^2 \;.
\end{align*}
By another application of this inequality, it even suffices to show that
\begin{align*}
   \theta(1-\theta)^2 \leq  {\theta} - \theta^2 \;,
\end{align*}
which reduces to $\theta \leq 1$. This inequality holds by the concavity of $\theta$, hence the proof is complete.
\end{proof}

 \appendix
\section{Some identities from non-commutative calculus}

Throughout this section we let $\Alg$ be the collection of $m \times m$-matrices with complex entries. The subset of self-adjoint elements shall be denoted by $\Alg_h$, and we let $\Alg_+$ be the collection of strictly positive elements in $\Alg$. 

For $x,y,z \in \Alg$ we consider the contraction operation $* : (\Alg \ot \Alg) \times \Alg \to \Alg$ defined by
 \begin{align}\label{eq:contr}
  (x \ot y)\contr z &:= xzy\;,
   \end{align}
and linear extension.

For a smooth function $f : (0,\infty) \to \R$ we define
 \begin{align*}
    \partial f(\lambda,\mu) := \left\{ \begin{array}{ll}
   \frac{f(\lambda) - f(\mu)}{\lambda-\mu}, \quad & \lambda \neq \mu\;,\\
  f'(\lambda),        \quad &\lambda = \mu\;.
\end{array} \right.
 \end{align*}
Let $X, Y \in \Alg_+$ with spectral decomposition 
 $X = \sum_{j=1}^m \lambda_j \hat x_j$ and $Y = \sum_{k=1}^m \mu_k \hat y_k$ for some $\lambda_j, \mu_k > 0$ and projections $\hat x_j, \hat y_k$ with 
$\sum_{j=1}^m \hat x_j = \sum_{k=1}^m \hat y_k = I$.
We define the \emph{non-commutative derivative} of $f$ as
 \begin{align*}
    \ncd f(X,Y) = \sum_{j,k =1}^m\partial f(\lambda_j,\mu_k)
    \hat x_j \ot \hat y_k\;.
  \end{align*}
  
The relevance of $\ncd f(X,Y)$ is due to the fact that it allows to formulate suitable versions of the chain rule in a non-commutative setting.

 \begin{proposition}\label{prop:chain-rule}
Let $f : (0,\infty) \to \R$ be a smooth function.
  \begin{enumerate}[(1)]
  \item(Discrete chain rule)
For $X, Y \in \Alg_+$ we have
 \begin{align} \label{eq:chain-diff-disc}
  f(X) - f(Y)  = \ncd f(X,Y) \contr (X - Y)\;.
 \end{align}
 \item(Chain rule)
For a smooth curve $t \mapsto X(t) \in \Alg_+$ we have
 \begin{align} \label{eq:chain-diff}
  \diff{}{t} f(X(t))  = \ncd f(X(t),X(t)) \contr X'(t)\;.
 \end{align}
 \end{enumerate}
 \end{proposition}

 \begin{proof}
To prove \eqref{eq:chain-diff-disc}, we write\begin{align*}
 f(X) - f(Y)
  & = \sum_{j,k=1}^{m} (f(\lambda_{j}) - f(\mu_{k})) \hat x_{j}\hat y_{k}
  \\&  = \sum_{j,k=1}^m \partial f(\lambda_j,\mu_k)(\lambda_{j} - \mu_{k})
       \hat x_j \hat y_k
  \\&  = \sum_{j,k=1}^m \partial f(\lambda_j,\mu_k)
       \hat x_j \ot  \hat y_k 
       \contr \bigg(\sum_{l,p=1}^m (\lambda_l - \mu_p) 
       \hat x_l  \hat y_p\bigg)
  \\& =  \ncd f(X,Y) \contr (X - Y)\;,
\end{align*}
where we used that $\hat x_j \hat x_l \hat y_p \hat y_k = \delta_{jl}\delta_{pk}\hat x_j \hat y_k$.

The identity \eqref{eq:chain-diff} is obtained by passing to the limit in \eqref{eq:chain-diff-disc}.
 \end{proof} 

It will be useful to compute the non-commutative derivatives of some frequently occurring functions. 
 
\begin{proposition}\label{prop:identities}
For $A,B \in \Alg_+$ we have 
  \begin{align*}
 \ncd[t \mapsto t^n](A,B)
   &= \sum_{j=0}^{n-1} A^{n-j-1} \ot B^{j}\;, \quad\quad\qquad\quad\qquad\qquad\qquad n =1,2,\ldots\;,\\
 \ncd[t \mapsto t^\alpha](A,B)
   &= \int_0^1 \int_0^\alpha 
  \frac{A^{\alpha - \beta}}{(1-s)I + s A}  \ot  
    \frac{B^\beta}{(1-s)I + s B}  \dd \beta \dd s\;, \quad \alpha \in (0,1)\;,\\
 \ncd\exp(A,B)
   & = \int_0^1 e^{(1-s)A} \ot e^{sB} \dd s\;,\\
   \ncd\log(A,B)
    & = \int_0^1
     ((1-s)I + s A )^{-1} \ot ((1-s)I + s B )^{-1} \dd s\;.
 \end{align*}
\end{proposition}
 
\begin{proof}
This follows from the following elementary identities, which hold for $\lambda, \mu > 0$:
\begin{align*}
   \ncd[t \mapsto t^n](\lambda,\mu)
    & = \sum_{l=0}^{n-1} \lambda^{n-l-1}\mu^l\;,\quad\quad\quad\qquad\qquad\qquad\qquad\qquad n =1,2,\ldots,\\
 \ncd[t \mapsto t^\alpha](\lambda,\mu)
   &= \int_0^1 \int_0^\alpha 
  \frac{\lambda^{\alpha - \beta}\mu^\beta}{((1-s) + s \lambda)((1-s) + s \mu)} 
 \dd \beta \dd s\;, \qquad \alpha \in (0,1)\;,\\
   \ncd\exp(\lambda, \mu)
    & = \int_0^1  e^{(1-t)  \lambda + t\mu} \dd s\;,\\
   \ncd\log(\lambda, \mu)
    & = \int_0^1 \frac{1}{
     ((1-s) + s \lambda) ((1-s) + s \mu )}  \dd s\;.
\end{align*}
\end{proof}

\medskip
 \noindent{\bf Acknowledgement}  This work was begun when both authors were visiting the Institute of Pure and Applied Mathematics at U.C.L.A. They would like to thank I.P.A.M. for its hospitality and support.

\bibliographystyle{plain}

 \end{document}